\documentclass[a4paper,reqno]{amsart}

\usepackage[a4paper,left=3cm,right=3cm,top=3cm,bottom=3cm]{geometry}
\usepackage{graphicx} 
\usepackage{tikz}
\usetikzlibrary{patterns,intersections,shapes,decorations.pathmorphing,decorations.pathreplacing,calc,arrows,cd}
\usetikzlibrary{tikzmark, fit, matrix, decorations.markings}
\usepackage{mathdots}
\usepackage{amssymb}
\usepackage{amstext}
\usepackage{amsmath}
\usepackage{amscd}
\usepackage{amsthm}
\usepackage{amsfonts}
\usepackage{rotating}

\usepackage{array}
\usepackage{enumerate}
\usepackage{latexsym}
\usepackage{mathrsfs}
\usepackage{hyperref}
\usepackage[all]{xy}
\usepackage{mathtools}
\usepackage{bm}
\usepackage{pgfplots}
\pgfplotsset{compat=1.16}

\usepackage{booktabs}
\usepackage{multirow}
\xyoption{all}
\usepackage{pstricks}
\usepackage{comment}

\usepackage{todonotes}

\tikzset{
    labl/.style={ rotate=30, inner sep=.5mm}
}
\tikzset{
    labll/.style={ rotate=-30, inner sep=.5mm}
}

\tikzset{
    lab90/.style={ rotate=90, inner sep=.5mm}
}

\numberwithin{equation}{section}
\newtheorem{theorem}{Theorem}[section]

\newtheorem{lemma}[theorem]{Lemma}
\newtheorem{proposition}[theorem]{Proposition}
\newtheorem{definition-theorem}[theorem]{Definition-Theorem}
\newtheorem{definition-proposition}[theorem]{Definition-Proposition}

\theoremstyle{definition}
\newtheorem{definition}[theorem]{Definition}

\newtheorem{remark}[theorem]{Remark}
\newtheorem{example}[theorem]{Example}

\newtheorem{construction}[theorem]{Construction}

\newtheorem{notation}[theorem]{Notation}

\newcommand{\Hom}{\operatorname{Hom}\nolimits}

\newcommand{\RHom}{\mathbf{R}\strut\kern-.2em\operatorname{Hom}\nolimits}

\newcommand{\Image}{\operatorname{Im}\nolimits}

\newcommand{\Hasse}{\operatorname{Hasse}\nolimits}

\DeclareMathOperator{\rep}{\it {\rm rep}_k}

\DeclareMathOperator{\vect}{\mathsf{vect}_{\it k}}

\newcommand{\Int}[1]{\mathbb{I}(#1)}

\newcommand{\smat}[1]{ \left[\begin{smallmatrix} #1 \end{smallmatrix}\right] }

\newcommand{\n}{\emptyset}
\newcommand{\0}{0}

\newcommand{\reltoeq}{\trianglerighteq}
\newcommand{\relto}{\vartriangleright}
\newcommand{\opreltoeq}{\trianglelefteq}
\newcommand{\oprelto}{\vartriangleleft}

\usepackage{yfonts}
\newcommand{\Qf}[1]{\textgoth{#1}}

\makeatletter
\def\Intp #1{\expandafter\Intp@i#1\@nil}
\def\Intp@i #1,#2\@nil{[\Qf{#1},\Qf{#2}]}

\def\itoi #1{\expandafter\itoi@i#1\@nil}
\def\itoi@i #1,#2,#3,#4\@nil{_{\Intp{#1,#2}}^{\Intp{#3,#4}}}
\makeatother

\makeatletter
\def\tikz@delimiter#1#2#3#4#5#6#7#8{%
  \bgroup
    \pgfextra{\let\tikz@save@last@fig@name=\tikz@last@fig@name}%
    node[outer sep=0pt,inner sep=0pt,draw=none,fill=none,anchor=#1,at=(\tikz@last@fig@name.#2),#3]
    {%
      {\nullfont\pgf@process{\pgfpointdiff{\pgfpointanchor{\tikz@last@fig@name}{#4}}{\pgfpointanchor{\tikz@last@fig@name}{#5}}}}%
      \delimitershortfall\z@
      \resizebox*{!}{#8}{$\left#6\vcenter{\hrule height .5#8 depth .5#8 width0pt}\right#7$}%
    }
    \pgfextra{\global\let\tikz@last@fig@name=\tikz@save@last@fig@name}%
  \egroup%
}
\makeatother

\newcommand{\mattikz}[1]{
   \begin{tikzpicture}[
      baseline = (p.center), 
      ampersand replacement=\&,
      decoration={
        markings,
        mark=
        at position 0.5
        with{
          \draw[-] (-2pt,-2pt) -- (2pt,2pt);
          \draw[-] (2pt,-2pt) -- (-2pt,2pt);
        }
      }]
      {#1}
    \end{tikzpicture}
}

\newcommand{\matheader}{
  \matrix[matrix of math nodes, column sep=0mm, row sep=0mm,
  inner sep = 0mm, left delimiter = {[},right delimiter = {]},
  every node/.append style={
    anchor=center,text depth = 0.375em,text
    height=0.875em,minimum width=1.25em}](p)
}
\newcommand{\matheaderwide}{
  \matrix[matrix of math nodes, column sep=0mm, row sep=0mm,
  inner sep = 0mm, left delimiter = {[},right delimiter = {]},
  every node/.append style={
    anchor=center,text depth = 0.75em,text
    height=0.875em,minimum width=2.4em}](p)
}

\usepackage{algpseudocode}
\usepackage{algorithm}
\algnewcommand{\LineComment}[1]{\State \(//\) #1}
\algnewcommand{\IfThenElse}[3]{
  \State \algorithmicif\ #1\ \algorithmicthen\ #2\ \algorithmicelse\ #3}
\algnewcommand{\ForDo}[2]
{ 
  \State \algorithmicforall\ #1\ \algorithmicdo\ #2
}

\makeatletter
\newcommand{\StatexIndent}[1][3]{%
  \setlength\@tempdima{\algorithmicindent}%
  \Statex\hskip\dimexpr#1\@tempdima\relax}
\makeatother

\date{\today}

\begin{document}

\title{Bipath Persistence}

\date{\today}

\author{Toshitaka Aoki}
\address{Toshitaka Aoki,
  Graduate School of Human Development and Environment, Kobe University, 3-11 Tsurukabuto, Nada-ku, Kobe 657-8501 Japan}
\email{toshitaka.aoki@people.kobe-u.ac.jp}

\author{Emerson G.\ Escolar}
\address{Emerson G.\ Escolar,
  Graduate School of Human Development and Environment, Kobe University, 3-11 Tsurukabuto, Nada-ku, Kobe 657-8501 Japan}
\email{e.g.escolar@people.kobe-u.ac.jp}

\author{Shunsuke Tada}
\address{Shunsuke Tada,
  Graduate School of Human Development and Environment, Kobe University, 3-11 Tsurukabuto, Nada-ku, Kobe 657-8501 Japan}
\email{205d851d@stu.kobe-u.ac.jp}


\begin{abstract}
  In persistent homology analysis,
  interval modules play a central role in describing the birth and death of topological features across a filtration.
  In this work, we extend this setting, and propose the use of \emph{bipath persistent homology},
  which can be used to study the persistence of topological features
  across a pair of filtrations connected at their ends, to compare the two filtrations.
  In this setting, interval-decomposability is guaranteed,
  and we provide an algorithm for computing persistence diagrams for bipath persistent homology
  and discuss the interpretation of bipath persistence diagrams.
\end{abstract}

\maketitle

\section{Introduction}

Persistent homology \cite{edelsbrunner2002topological, carlsson2005computing} has been successfully applied in diverse domains, such as material science \cite{hiraoka2016hierarchical}, 
evolutionary biology \cite{carlsson2013evolution}, 
and others \cite{sousbie1_2011persistent,sousbie2_2011persistent,aktas2019persistence}.
Let us recall how persistent homology arises in a typical way. 
For point cloud data,
one can construct a filtration of simplicial complexes
\[
S\colon\  S_1\subseteq S_2 \subseteq \cdots \subseteq S_n
\]
for example by using the \v{C}ech complex, the Vietoris-Rips complex, or the alpha complex construction.
More generally, one can consider a filtration of topological spaces.
Applying the $q$th degree homology functor $H_q(-;k)$ (with coefficient field $k$)
to the filtration 
$S$,
we obtain a \emph{persistence module} 
\[
 H_q(S;k)\colon \  H_q(S_1 ;k) \to H_q(S_2;k) \cdots \to  H_q(S_n ;k) 
\]
where each linear map $H_q(S_a ;k) \to H_q (S_b;k)$
is induced by the inclusion map $S_a \xhookrightarrow{} S_b$.
A structure theorem (Gabriel’s theorem \cite{Gabriel72}, etc.) guarantees that
the persistence module $H_q(S;k)$ decomposes into interval modules.
Going back to the filtration, these intervals $[b,d)$ are then interpreted
as describing some topological feature's persistence (born at $b$, dies at $d$).
Thus, the 
interval modules (and interval-decomposition) play a central role in describing the birth and death of topological features
(e.g.\ connected components, rings, cavities, etc) of data, and is visualized as a  \emph{persistence diagram}.

In order to accomodate different diagrams of spaces arising from data,
the above standard setting has been naturally generalized in various ways,
for example, zigzag persistence \cite{carlsson2010zigzag}
or multiparameter persistence \cite{carlsson2009theory}.
It is a fact that zigzag persistence modules are also guaranteed to admit interval-decomposability, 
while multiparameter persistence modules are not in general.
For their motivation and applications, see
\cite{carlsson2009zigzag,tausz2011applications,myers2023temporal,mcdonald2023topological}
and \cite{carlsson2009theory,botnan2022introduction}, respectively, for example.

In this work, we propose a new method for studying the persistence of topological features
across a family of spaces along a bipath, or a pair of filtrations connected at their ends, as below:
\begin{equation*}
  S\colon \ 
  \begin{tikzcd}[row sep=0.1em,column sep = 1.4em, inner sep=0pt]
    & S_1 \rar[hookrightarrow] & S_2 \rar[hookrightarrow] & \cdots \rar[hookrightarrow] & S_n \ar[dr,hookrightarrow] & \\
    S_{\hat0} \ar[ur,hookrightarrow] \ar[dr,hookrightarrow] & & & &  & S_{\hat1} \\
    & S_{1'} \rar[hookrightarrow] & S_{2'} \rar[hookrightarrow] & \cdots \rar[hookrightarrow] & S_{m'} \ar[ur,hookrightarrow] &
  \end{tikzcd}.
\end{equation*}
We call such a diagram of spaces a \emph{bipath filtration}.
From its definition,
a bipath filtration
can be thought of as two filtrations    
sharing the same starting space $S_{\hat0}$
and ending space $S_{\hat1}$ but going through possibly different processes in the middle.

This leads us to consider persistence modules over bipath posets, or \emph{bipath persistence modules} for short.
We provide three motivations to consider bipath persistence modules.

First, and from a theoretical point of view,
the bipath persistence modules are a special class of persistence modules
that arises naturally when considering interval-decomposability.
This follows from the fact (see for example \cite[Theorem 1.3]{aoki2023summand})
that a bipath persistence module is guaranteed to be interval-decomposable,
and thus we can easily define its persistence diagram.

In fact, the result of \cite[Theorem 1.3]{aoki2023summand}, states
that for a finite connected poset $P$, every $P$-persistence module is interval decomposable
is equivalent to $P$ being either an $A$-type poset or a bipath poset.
Noting that
persistence modules over an $A$-type poset are exactly the zigzag persistence modules,
the result \cite[Theorem 1.3]{aoki2023summand}
says that apart from zigzag persistence (including classical persistence),
the only other setting for finite posets for which interval-decomposability is guaranteed is our setting of bipaths.
Thus, our second reason for considering bipaths is
that it fills a gap in the literature in exploiting interval-decomposability.

Finally, and perhaps most importantly,
bipath persistence modules can be used to study the persistence of topological features
across a pair of filtrations connected at their ends, to compare the two filtrations.
This setting can be compared to the setting of a filtration over a commutative ladder \cite{escolar2016persistence},
which connects two filtrations at every point. However, persistence modules over commutative ladders can be
very complicated for sufficiently long ladders \cite[Theorem~4]{escolar2016persistence},
in contrast to the bipath persistence modules which are always interval-decomposable.
Furthermore, bipath persistence modules are naturally obtained by restriction of multiparameter persistence modules
(see Example \ref{ex:embedding_bipath}), and thus can be thought of as a simpler version of multiparameter persistence.

In this paper, we propose methods for decomposing bipath persistence modules.
In Section~\ref{Section:Dim=0},
we provide a straightforward way for computing 
a decomposition by changing the bases of vector spaces consisting bipath persistence modules.
In Section~\ref{section:matrixmethod},
we provide our main algorithm, which uses the idea of so-called matrix problems methods
(see for example \cite[Chapter 1]{gabriel1992representations}, \cite{asashiba2019matrix})
to compute a decomposition.
We use column and row operations as in elementary linear algebra,
but only certain operations are deemed permissible as determined by the structure of
homomorphisms between intervals.
Using these permissible operations, we find a normal form corresponding to an interval decomposition of the input bipath persistence module.
Assuming that the bipath persistence module is already in a specific form,
or can be easily transformed to a such a form,
say for example starting from a bipath filtration of simplicial complexes,
the algorithm in Section~\ref{section:matrixmethod}
is more efficient (see the last part of Subsection~\ref{subsec:algosummary}).
Nevertheless,
we provide the method in Section~\ref{Section:Dim=0}
because of its simplicity and the fact that it 
highlights a connection with zigzag persistence. 

In addition, thanks to interval-decomposability we are able to define persistence diagrams for bipath persistence modules, which we call \emph{bipath persistence diagram}. 
We remark that our bipath persistence diagram naturally includes the standard persistence diagram. 
We define it in Definition~\ref{def:barcodeV}
and provide examples in Section \ref{sec:bipathPD}.

\subsection{Related work}

In \cite{burghelea2011persistence}, persistence of circle valued maps, which can be understood as quiver representations over a bipartite quiver of type $\Tilde{A}$, is studied. In the paper, algorithms for the decomposition of those representations are proposed. We remark that bipath persistence modules and quiver representations over $\Tilde{A}$ have different algebraic structures, and therefore, their algorithm is not applicable as-is to our bipath persistence modules.

On the other hand, all indecomposable modules over bipath posets are classified up to isomorphisms by using string combinatorics and it is shown that they are exactly interval modules by \cite{aoki2023summand} (see \cite{erdmann2006blocks} for basic knowledge of string algebras). 
In contrast to our method (and purpose) of this paper, such a classification result does not imply giving algorithms of decomposition into indecomposable modules.


\section{Preliminaries}
In this section we recall the notion of persistence modules. 
In particular, we study persistence modules over $A$-type posets and bipath posets
which naturally arise from the point of view of interval-decomposability (Theorem~\ref{thm:intervaldecomposable_iff}),
and discuss their persistence diagrams.

\subsection{Persistence modules}\label{sec:pers_modules}
We recall basic definitions of persistence modules  \cite{botnan2020decomposition, bubenik2014categorification, chazal2013structure}.  
Let $k$ be a field. 
Let $P$ be a (not necessarily finite) poset equipped with a partial order $\leq$. 
We regard a given poset $P$ as a category in such a way that its objects are elements of $P$ and 
there exists a unique morphism from $a$ to $b$ whenever $a\leq b$. 
For a poset $P$, we denote the category of (covariant) functors from $P$ to the category $\vect$ of finite dimensional $k$-vector spaces by $\rep(P)$, which is abelian. 
Each object of $\rep(P)$ is called a \emph{$P$-persistence module}. 
More explicitly, objects and morphisms of this category are given as follows. 
A $P$-persistence module $V$ is a correspondence such that
\begin{itemize}
\item it associates each element $a \in P$ to a finite dimensional $k$-vector space $V(a)$, and
\item it associates each relation $a\leq b$ in $P$ to a $k$-linear map $V(a \leq b)\colon V(a) \to V(b)$ such that 
  $V(a \leq a)=1_{V(a)}$ and 
  $V(b\leq c)\circ V(a\leq b)=V(a\leq c)$.
\end{itemize}

For two $P$-persistence modules $V$ and $W$, 
a morphism $f\colon V \to W$ is a family of $k$-linear maps $f_a : V(a) \to W(a)$ indexed by $a \in P$ and satisfying 
$f_b \circ V(a \leq b) = W(a \leq b)\circ f_a$ 
for every relation $a\leq b$ in $P$, that is, the following diagram is commutative.
\[
\xymatrix{
    V(a) \ar[rr]^-{V(a \leq b)} \ar[d]^{f_a}
    & &V(b) \ar[d] ^{f_b}
    \\
    W(a) \ar[rr]^-{W(a \leq b)}
    & \ar@{}[u]|{\circlearrowright} &W(b)
}
\]
A morphism $f$ is called an \emph{isomorphism} if
all the $f_a$'s are isomorphisms of $k$-vector spaces.
If there is an isomorphism $f: V\rightarrow W$, we write $V\cong W$. 

A \emph{direct sum} $V\oplus W$ of two $P$-persistence modules $V$ and $W$ is provided by 
\begin{equation*}
    (V\oplus W)(a) := V(a) \oplus W(a) \quad \text{and} \quad 
    (V\oplus W)(a\leq b) := 
    \begin{bmatrix}
        V(a\leq b) & 0 \\ 
        0 & W(a\leq b)  
    \end{bmatrix}. 
\end{equation*}
We say that $V$ is \emph{decomposable} if it can be written as a direct sum of non-zero objects, and \emph{indecomposable} otherwise. 
The Krull-Schmidt theorem asserts that every $P$-persistence module decomposes into indecomposable objects uniquely, up to isomorphism and order of summands.

In the rest of this subsection, let $P$ be a finite poset. 
We denote by $\Hasse(P)$ the \emph{Hasse quiver}  of $P$, that is, a quiver whose set of vertices is $P$, and having an arrow $a \to b$ for $a,b\in P$ if and only if $a < b$ and there is no $z\in P$ such that $a < z < b$. 
We remark that any finite poset is determined by its Hasse quiver. We also remark that it suffices to specify the vector spaces $V(a)$ for all $a \in P$ and the linear maps $V(a \leq b)$ for arrows $a \to b$ in $\Hasse(P)$ to uniquely specify a $P$-persistence module $V$.

Now, we define intervals an interval modules over finite posets. 

\begin{definition}\label{def:conv_interval}
    Let $I\subseteq P$ be a subset. 
    We say that $I$ is 
    \begin{itemize}
        \item \emph{convex} if, for any $a, b \in I$ and $z \in P$, $a < z < b$ implies $z \in  I$. 
        \item \emph{connected} if, for any $a, b \in I$, there exists a sequence $a=z_0,z_1.\ldots,z_l=b$ of elements of $I$ such that $z_{i-1}$ and $z_{i}$ are comparable for every $i\in \{1,\ldots,l\}$.  
        \item \emph{interval} if it is convex and connected in $P$. 
    \end{itemize}
    We denote by $\mathbb{I}(P)$ the set of intervals of $P$. 
\end{definition}

\begin{definition}\label{def:intervalmod}
    For an interval $I\in \mathbb{I}(P)$, 
    let $k_I$ be a $P$-persistence module defined as follows.
    \begin{equation}
        k_I(a) =
        \begin{cases}
        k & \text{if $a\in I$}, \\
        0 & \text{otherwise,}
        \end{cases}\quad \quad
        k_I(a\leq b) =
        \begin{cases}
        1_k & \text{if $a,b\in I$}, \\
        0 & \text{otherwise.}
        \end{cases}
    \end{equation}

    A $P$-persistence module $M$ is said to be \emph{interval} if $M\cong k_I$ for some interval $I \in \mathbb{I}(P)$, 
    and \emph{interval-decomposable} if it is isomorphic to a direct sum of finite copies of some interval modules.
    Clearly, all interval modules are indecomposable. 
\end{definition}

In \cite{aoki2023summand}, they classified finite posets whose persistence module are always interval-decomposable. 
This is done by using $A$-type posets and bipath posets defined as follows.   

\begin{definition} \label{def:ab_posets}
\begin{enumerate}[\rm (1)]
    \item For an integer $n>0$, let $A_n(a)$ be a poset whose Hasse quiver is of the form 
\begin{equation} \label{eq:A-type}
    A_n(a)\colon \ 1 \longleftrightarrow 2  \longleftrightarrow \cdots \cdots \longleftrightarrow n,
\end{equation}
where $\leftrightarrow$ is either $\rightarrow$ or $\leftarrow$ as assigned by the orientation $a$. 
We call them \emph{$A$-type posets}. 

The \emph{equioriented} $A$-type poset $A_n(e)$ is provided by an orientation 
    \begin{equation} \label{eq:equioriented}
       A_n(e)\colon \  1 \longrightarrow 2 \longrightarrow \cdots \cdots \longrightarrow n. 
    \end{equation}
\item For two integers $n,m>0$, let $B_{n,m}$ be a poset given by adding two distinguished points $\hat{0}$ and $\hat{1}$ to a disjoint union of $A_n(e)$ and $A_m(e)$ so that $\hat{0}$ (resp., $\hat{1}$) is a global minimum (resp., maximum). 
By definition, the Hasse quiver of $B_{n,m}$ is given by
\begin{equation}\label{eq:HasseB}
B_{n,m}\colon \ 
\begin{tikzcd}[row sep=0.1em, column sep = 1.2em, inner sep=0pt]
& 1  \rar & 2  \rar & \cdots \rar & n \ar[dr] & \\
\hat{0} \ar[ur] \ar[dr] & & & &  & \hat{1} \\
& 1' \rar & 2' \rar & \cdots \rar & m' \ar[ur] &
\end{tikzcd}    
\end{equation}
where we label elements of $A_{n}(e)$ (resp., $A_{m}(e)$) by $1,\ldots,n$ (resp., $1',\ldots,m'$). 
We call $B_{n,m}$ a \emph{bipath poset}. 
\end{enumerate}
\end{definition}

\begin{theorem}{\rm \cite[Theorem 1.3]{aoki2023summand}} 
\label{thm:intervaldecomposable_iff}
    For a finite connected poset $P$,   
    the following conditions are equivalent.
    \begin{enumerate}[\rm (a)]
        \item Every $P$-persistence module is interval-decomposable. 
        \item $P$ is either an $A$-type poset or a bipath poset. 
    \end{enumerate}
\end{theorem}

\subsection{Persistence diagrams} 
\label{sec:bipath_modules}

In this section, we discuss the persistence diagrams which give a practical tool for persistent homology.

\subsubsection{Zigzag persistence modules}
Firstly, we study an $A$-type poset $A := A_n(a)$, where $a$ is an arbitrary orientation. 
$A$-persistence modules are known as \emph{zigzag persistence modules}.  
From its definition, intervals of $A$ are exactly the subsets of the form $[b,d]:=\{b,b+1,\ldots,d\}$ for some $1\leq b \leq d \leq n$. Namely, 
\begin{equation}\label{eq:Atype_intervals}
    \mathbb{I}(A) = \{[b,d] \mid 1\leq b \leq d \leq n\}. 
\end{equation}

For $1\leq b \leq d \leq n$, the interval $[b,d]$ and the corresponding interval module $k_{[a,b]}$ can be described in the following picture, where $\leftrightarrow$ is either $\rightarrow$ or $\leftarrow$ as assigned by the orientation $a$.
\begin{eqnarray*}
\begin{tikzpicture}
    \coordinate (b) at (0,0);
    \coordinate (d) at (6,0);
    \coordinate (x) at (1,0);
    \coordinate (z) at (0,-0.8); 
    \coordinate (dd) at ($(d)+ -1*(x)$);
    \draw[line width = 0.25mm] ($(b) + -2*(x)$)--($(dd) + 2*(x)$); 
    \node (lcdots) at ($(b) + -2.5*(x)$) {$\cdots$}; 
    \node (rcdots) at ($(dd) + 2.5*(x)$) {$\cdots$}; 
    \node at ($(b)+(-0,0)$) [above]{\small$b$}; 
    \node at (dd) [above]{\small$d$}; 
    \draw[line width = 0.5mm] (b)--(dd); 
    \fill[] (b)circle(0.8mm);
    \fill[] (dd)circle(0.8mm);
    \node (p) at ($(b) + (-4.2,0)$) {$[b,d]$:};
    \node (pp) at ($(b) + (-4.2,0) + (z)$) {$k_{\mbox{\small$[b,d]$}}$:};
    \node (bk) at ($(b)+(z)$) {$k$}; 
    \node (dk) at ($(d)+(z)$) {$0$}; 
    \node (d-1k) at ($(d) + -1*(x) + (z)$) {$k$}; 
    \node (b-1k)  at ($(b)+ -1*(x) + (z)$) {$0$};
    \node (b+1k)  at ($(b)+ (x) + (z)$) {$k$};
    \node (d+1k) at ($(d) + (x) + (z)$) {}; 
    \node (b+2k)  at ($(b)+ 2*(x) + (z)$) {};
    \node (d-2k) at ($(d) + -2*(x)+(z)$) {$k$}; 
    \node (d-3k) at ($(d) + -3*(x)+(z)$) {}; 
    \node (c) at ($(bk)!0.5!(d-1k)$) {$\cdots$}; 
    \node (d+2k) at ($(d) + 2*(x)+(z)$) {}; 
    \node (b-2k) at ($(b) + -2*(x)+(z)$) {}; 
    \node (d+3k) at ($(dd) + 2.5*(x) +(z)$) {$ \cdots$}; 
    \node (b-3k) at ($(b) + -2.5*(x) +(z)$) {$\cdots$}; 
    
    \draw[<->, ] (d-1k)--(dk);
    \draw[<->, ] (b-1k)--(bk);
    \draw[<->, ] (dk)--(d+1k);
    \draw[<->, ] (b-2k)--(b-1k);
    
    \draw[<->, ] (bk)--node[above]{\footnotesize$1_k$}(b+1k);
    \draw[<->, ] (d-2k)--node[above]{\footnotesize$1_k$}(d-1k);
    \draw[<->, ] (b+1k)--node[above]{\footnotesize$1_k$}(b+2k);
    \draw[<->, ] (d-3k)--node[above]{\footnotesize$1_k$}(d-2k);
\end{tikzpicture}
\end{eqnarray*}

As we mentioned before, every $A$-persistence module $V$ is interval-decomposable. Thus, there exists an isomorphism 
\begin{equation} \label{eq:str_Atype}
    V \cong \bigoplus_{I\in \mathbb{I}(A)} k_{I}^{m(I)} 
\end{equation}
in $\rep(A)$, where $m(I)$ denotes the multiplicity of the interval module $k_I$ as its summands. Notice that such a decomposition is unique up to isomorphism and order of summands.
In this case, we define the \emph{persistence diagram} of $V$ by a multiset 
\begin{equation*}
    \{\text{$I \in \mathbb{I}(A)$ with multiplicity $m(I)$}\}.  
\end{equation*}
When $A$ is equioriented, we call it the \emph{standard persistence diagram}.  

Using the description \eqref{eq:Atype_intervals}, 
this admits a realization in the plane $\mathbb{R}^2$ by
(see \cite{carlsson2010zigzag,otter2017roadmap,oudot2017persistence} for example)
\begin{equation*}
    \{\text{$(a,b) \in \mathbb{R}^2$ for $I=[a,b]$ with multiplicity $m(I)$}\}. 
\end{equation*}

\subsubsection{Bipath persistence modules} \label{sec:bipath_pd}
Next, we study persistence modules over bipath posets (or \emph{bipath persistence modules} for short). 
Let $B:=B_{n,m}$ be a bipath poset as in Definition~\ref{def:ab_posets}(2) for fixed $n,m>0$. 
In terms of its Hasse quiver, we recall that $V\in \rep(B)$ is determined by a diagram of vector spaces and linear maps  
\[    
    \begin{tikzcd}[row sep=0.1em, column sep=4.5em, inner sep=0pt]
      & V(1) \rar{V(1\leq 2)} & V(2) \rar{V(2\leq 3)} & \cdots \rar{V(n-1\leq n)} & V(n) \ar[dr]{}{V(n\leq \hat{1})} & \\ 
      V(\hat{0}) \ar[ur]{}{V(\hat{0} \leq 1)} 
      \ar[dr]{}{V(\hat{0}\leq 1')} & & & & & V(\hat{1}) \\    
      & V(1') \rar{V(1'\leq 2')} & V(2') \rar{V(2'\leq 3')} & \cdots \rar{V(m-1' \leq m')} & V(m') \ar[ur]{}{V(m'\leq \hat{1})} &
    \end{tikzcd}
\]
where they need to satisfy the commutative relation
\begin{equation}\label{eq:commutativity}
    V(n\leq \hat{1}) \cdots V(1\leq 2) V(\hat{0}\leq 1) =
    V(m'\leq \hat{1}) \cdots V(1'\leq 2') V(\hat{0}\leq 1'). 
\end{equation}

To define its persistence diagram, 
we begin with enumerating all intervals over $B$. 
Now, let $U$ and $D$ be the upper path and the lower path of $B$ respectively.
\begin{equation}\label{eq:UandD}
    U \colon \ \hat{0} \to 1 \to 2 \to \cdots \to n \to \hat{1} \quad \text{and} \quad     
    D \colon \ \hat{0} \to 1' \to 2' \to \cdots \to m' \to \hat{1}. 
\end{equation}

\begin{definition} \label{def:LRUD}
With the above notation, we define the following sets of intervals of $B$. 
\begin{itemize}
    \item The set $\mathbb{L}(B)$ consists of 
    $[\hat{0},t] \cup [\hat{0},s]$ for some $t\in U$ and $s\in D$ with $s,t\neq \hat{1}$. 
    \item The set $\mathbb{R}(B)$ consists of 
    $[s, \hat{1}] \cup [t, \hat{1}]$ for some $s\in U$ and $t\in D$ with $s,t\neq \hat{0}$. 
    \item The set $\mathbb{U}(B)$ consists of 
    $[s, t]$ for some $s,t\in U$ with $s,t\not \in \{\hat{0},\hat{1}\}$. 
    \item The set $\mathbb{D}(B)$ consists of 
    $[t,s]$ for some $s,t\in D$ with $s,t\not \in \{\hat{0},\hat{1}\}$. 
\end{itemize}
Here, the symbol $[a,b]$ means the interval considered in the $A$-type poset $U$ or $D$. 
If an interval $J$ of $B$ is given by one of the above form using $s$ and $t$, then we write $J= \langle s,t \rangle$. 
\end{definition}

Our notation $\langle s,t \rangle$ for intervals can be understood 
as the walk from $s$ to $t$ in a clockwise direction in the Hasse quiver of $B$. 
See Table~\ref{tab:intervalsB}. 
For instance, the interval $[t,s]$ of the lower path $D$ corresponds to the interval $\langle s,t \rangle$ of $B$.
We explain our reasoning for this in Remark~\ref{rem:interpret}, after introducing the persistence diagram.

\begin{table}[h]
\renewcommand{\arraystretch}{1.5}
    \begin{tabular}{cccccccc}
        $B$ & \hspace{-2mm}$\mathbb{L}(B)$ & \hspace{-2mm}$\mathbb{R}(B)$ & \hspace{-2mm} $\mathbb{U}(B)$& \hspace{-2mm}$\mathbb{D}(B)$  \\ 
        \begin{tikzpicture}[baseline = 0mm, scale = 0.9]
        \coordinate (x) at (0.4,0);
        \coordinate (y) at (0,0.6); 
        \coordinate (0) at ($0*(x) + 0*(y)$);
        \coordinate (1) at ($1*(x) + 1*(y) - 0.1*(x)$);
        \coordinate (n) at ($5*(x) + 1*(y)$);
        \coordinate (11) at ($6*(x) + 0*(y) - 0.1*(x)$);
        \coordinate (1') at ($1*(x) + -1*(y) - 0.1*(x)$);
        \coordinate (m') at ($5*(x) + -1*(y)$);
        \draw (0)--(1)--(n)--(11)--(m')--(1')--cycle;
        \node[left] at (0) {\footnotesize$\hat{0}$};
        \node[right] at (11) {\footnotesize$\hat{1}$}; 
        \node[above] at (1) {\footnotesize$1$}; 
        \node[above] at (n) {\footnotesize$n$}; 
        \node[below] at (1') {\footnotesize$1'$}; 
        \node[below] at (m') {\footnotesize$m'$}; 
        \draw[line width = 0.7mm] (0)--(1)--(n)--(11)--(m')--(1')--cycle;
        \end{tikzpicture} 
        & \hspace{-4mm}
        \begin{tikzpicture}[baseline = 0mm, scale =0.9]
        \coordinate (x) at (0.4,0);
        \coordinate (y) at (0,0.6); 
        \coordinate (0) at ($0*(x) + 0*(y)$);
        \coordinate (1) at ($1*(x) + 1*(y) - 0.1*(x)$);
        \coordinate (n) at ($5*(x) + 1*(y)$);
        \coordinate (11) at ($6*(x) + 0*(y) - 0.1*(x)$);
        \coordinate (1') at ($1*(x) + -1*(y) - 0.1*(x)$);
        \coordinate (m') at ($5*(x) + -1*(y)$);
        \draw (0)--(1)--(n)--(11)--(m')--(1')--cycle;
        
        \coordinate (s) at ($3*(x) + -1*(y)$);
        \coordinate (t) at ($3*(x) + 1*(y)$);
        \fill (s)circle(0.7mm)node[below]{$s$};
        \fill (t)circle(0.7mm)node[above]{$t$};
        \draw[line width = 0.7mm] (s)--(1')--(0)--(1)--(t);
        \node[left] at (0) {\footnotesize$\hat{0}$};
        \node[right] at (11) {\footnotesize$\hat{1}$}; 
        \node[above] at (1) {\footnotesize$1$}; 
        \node[above] at (n) {\footnotesize$n$}; 
        \node[below] at (1') {\footnotesize$1'$}; 
        \node[below] at (m') {\footnotesize$m'$}; 
        
        \end{tikzpicture} 
        & \hspace{-4mm}
        \begin{tikzpicture}[baseline = 0mm, scale =0.9]
        \coordinate (x) at (0.4,0);
        \coordinate (y) at (0,0.6); 
        \coordinate (0) at ($0*(x) + 0*(y)$);
        \coordinate (1) at ($1*(x) + 1*(y) - 0.1*(x)$);
        \coordinate (n) at ($5*(x) + 1*(y)$);
        \coordinate (11) at ($6*(x) + 0*(y) - 0.1*(x)$);
        \coordinate (1') at ($1*(x) + -1*(y) - 0.1*(x)$);
        \coordinate (m') at ($5*(x) + -1*(y)$);
        \draw (0)--(1)--(n)--(11)--(m')--(1')--cycle;
        
        \coordinate (s) at ($3*(x) + 1*(y)$);
        \coordinate (t) at ($3*(x) + -1*(y)$);
        \fill (s)circle(0.7mm)node[above]{$s$};
        \fill (t)circle(0.7mm)node[below]{$t$};
        \draw[line width = 0.7mm] (s)--(n)--(11)--(m')--(t);
        \node[left] at (0) {\footnotesize$\hat{0}$};
        \node[right] at (11) {\footnotesize$\hat{1}$}; 
        \node[above] at (1) {\footnotesize$1$}; 
        \node[above] at (n) {\footnotesize$n$}; 
        \node[below] at (1') {\footnotesize$1'$}; 
        \node[below] at (m') {\footnotesize$m'$}; 
        
        \end{tikzpicture} 
        & \hspace{-4mm}
        \begin{tikzpicture}[baseline = 0mm, scale =0.9]
        \coordinate (x) at (0.4,0);
        \coordinate (y) at (0,0.6); 
        \coordinate (0) at ($0*(x) + 0*(y)$);
        \coordinate (1) at ($1*(x) + 1*(y) - 0.1*(x)$);
        \coordinate (n) at ($5*(x) + 1*(y)$);
        \coordinate (11) at ($6*(x) + 0*(y) - 0.1*(x)$);
        \coordinate (1') at ($1*(x) + -1*(y) - 0.1*(x)$);
        \coordinate (m') at ($5*(x) + -1*(y)$);
        \draw (0)--(1)--(n)--(11)--(m')--(1')--cycle;
        
        \coordinate (s) at ($2*(x) + 1*(y)$);
        \coordinate (t) at ($4*(x) + 1*(y)$);
        \fill (s)circle(0.7mm)node[above]{$s$};
        \fill (t)circle(0.7mm)node[above]{$t$};
        \draw[line width = 0.7mm] (s)--(t);
        \node[left] at (0) {\footnotesize$\hat{0}$};
        \node[right] at (11) {\footnotesize$\hat{1}$}; 
        \node[above] at (1) {\footnotesize$1$}; 
        \node[above] at (n) {\footnotesize$n$}; 
        \node[below] at (1') {\footnotesize$1'$}; 
        \node[below] at (m') {\footnotesize$m'$}; 
        
        \end{tikzpicture} 
        & \hspace{-4mm}
        \begin{tikzpicture}[baseline = 0mm, scale =0.9]
        \coordinate (x) at (0.4,0);
        \coordinate (y) at (0,0.6); 
        \coordinate (0) at ($0*(x) + 0*(y)$);
        \coordinate (1) at ($1*(x) + 1*(y) - 0.1*(x)$);
        \coordinate (n) at ($5*(x) + 1*(y)$);
        \coordinate (11) at ($6*(x) + 0*(y) - 0.1*(x)$);
        \coordinate (1') at ($1*(x) + -1*(y) - 0.1*(x)$);
        \coordinate (m') at ($5*(x) + -1*(y)$);
        \draw (0)--(1)--(n)--(11)--(m')--(1')--cycle;
        
        \coordinate (s) at ($4*(x) + -1*(y)$);
        \coordinate (t) at ($2*(x) + -1*(y)$);
        \fill (s)circle(0.7mm)node[below]{$s$};
        \fill (t)circle(0.7mm)node[below]{$t$};
        \draw[line width = 0.7mm] (t)--(s);
        \node[left] at (0) {\footnotesize$\hat{0}$};
        \node[right] at (11) {\footnotesize$\hat{1}$}; 
        \node[above] at (1) {\footnotesize$1$}; 
        \node[above] at (n) {\footnotesize$n$}; 
        \node[below] at (1') {\footnotesize$1'$}; 
        \node[below] at (m') {\footnotesize$m'$}; 
        
        \end{tikzpicture} 
    \end{tabular}
    \caption{All intervals of $B$. 
    The left most one is $B$ itself. For each 
    $\mathbb{X} \in \{\mathbb{L},\mathbb{R},\mathbb{U},\mathbb{D}\}$, 
    we describe intervals $J=\langle s,t \rangle\in \mathbb{X}(B)$ by thick lines. 
    The corresponding interval modules are given by putting 
    the $1$-dimensional vector spaces $k$ on $J$, 
    the identity map $1_k$ between them, and $0$ elsewhere. 
    }
    \label{tab:intervalsB}
\end{table}

Under the above notation, we have the following result. 

\begin{proposition} \label{prop:str_bipathmodule}
Let $B:= B_{n,m}$ be a bipath poset. 
\begin{enumerate}[\rm (1)]
    \item We have 
    $\mathbb{I}(B) = \{B\} \sqcup \mathbb{L}(B) \sqcup \mathbb{R}(B)\sqcup \mathbb{U}(B)\sqcup \mathbb{D}(B)$. 
    \item For any $V\in \rep(B)$, we have a decomposition 
    $V = V_{B} \oplus V_{\mathbb{L}} \oplus V_{\mathbb{R}} \oplus V_{\mathbb{U}}\oplus V_{\mathbb{D}}$, 
    where 
    \begin{equation*}
        V_{B} \cong k_B^{m(B)} \quad \text{and} \quad 
        V_{\mathbb{X}} \cong \bigoplus_{J\in \mathbb{X}(B)} k_J^{m(J)} \ \ 
        \text{for 
        $\mathbb{X}\in 
        \{\mathbb{L},\mathbb{R},\mathbb{U},\mathbb{D}\}.$}
    \end{equation*}
\end{enumerate}
\end{proposition}

\begin{proof}
    It is immediate from \cite[Proposition~5.5 and Corollary~5.7]{aoki2023summand}. 
\end{proof}

Thanks to interval-decomposability, we can define a persistence diagram for every bipath persistence module as follows. 

\begin{definition}\label{def:barcodeV}
    Suppose that a bipath persistence module $V$ is decomposed as in Proposition~\ref{prop:str_bipathmodule}(2). 
    The \emph{persistence diagram} of $V$ is defined to be a multiset 
    \begin{eqnarray*}
        \mathcal{D}(V) &:=& \{\text{$J\in \mathbb{I}(B)$ with multiplicity $m(J)$}\} \\ 
        &=& \mathcal{D}(V_{B}) \sqcup \mathcal{D}(V_{\mathbb{L}}) \sqcup \mathcal{D}(V_{\mathbb{R}}) \sqcup \mathcal{D}(V_{\mathbb{U}}) \sqcup \mathcal{D}(V_{\mathbb{D}}). 
    \end{eqnarray*}
\end{definition}

It is natural to ask a visualization of persistence diagrams of bipath persistence modules similar to the case of $A$-type posets. 
To do this, we adopt the following convention for it in this paper.  

\bigskip \noindent
\underline{{\bf Visualization of persistence diagrams.}} 
We describe the persistence diagram $\mathcal{D}(V)$ of a given bipath persistence module $V$ using Figure \ref{fig:bipathPD}. 
Namely, we use the diagram provided by the following procedure: 
\begin{itemize} 
\item Plot points at the upper left corner of the Figure \ref{fig:bipathPD} with multiplicity $m(B)$, to correspond to the intervals $B$.

\item Plot points $(s,t)$ with multiplicity $m(\langle s,t \rangle)$ to correspond to the intervals $\langle s,t \rangle$.
\end{itemize}
We will denote it by $p(\mathcal{D}(V))$. 
By definition, it can be divided into $5$ components as 
\begin{equation}
        p(\mathcal{D}(V)) = p(\mathcal{D}(V_{B})) \sqcup p(\mathcal{D}(V_{\mathbb{L}})) \sqcup p(\mathcal{D}(V_{\mathbb{R}})) \sqcup p(\mathcal{D}(V_{\mathbb{U}})) \sqcup p(\mathcal{D}(V_{\mathbb{D}})),  
    \end{equation}
where the corresponding regions are explained and visualized in Figure~\ref{fig:bipathPD}. 

\begin{figure}[htbp]
\begin{center}
\includegraphics[width=120mm]{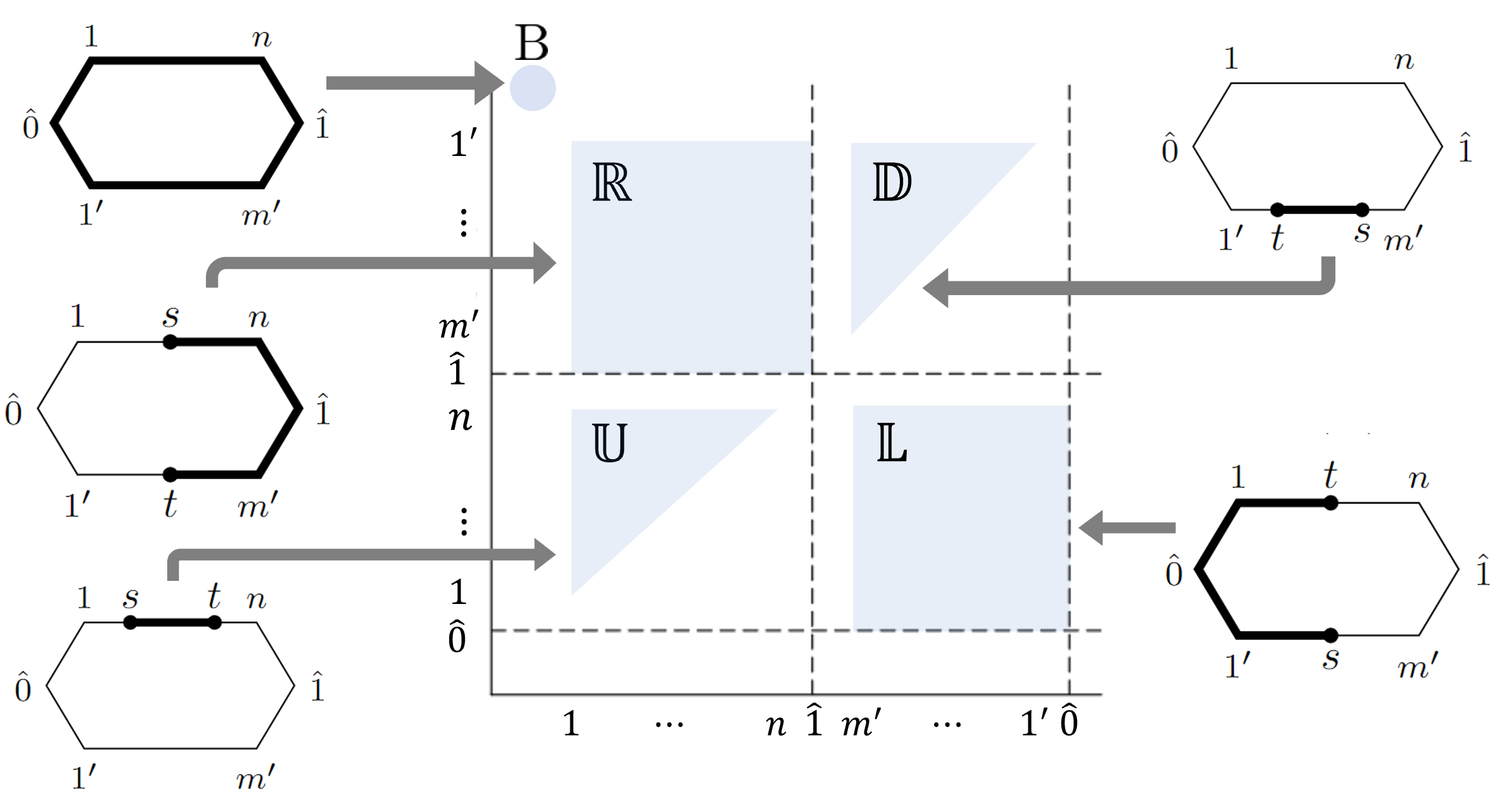}
\caption{A diagram for a visualization of persistence diagrams. 
The interval $B$ corresponds to region $\rm B$ (i.e., the point) in the upper left corner, 
and each interval $\langle s,t \rangle\in \mathbb{X}(B)$
corresponds to a point $(s,t)$ in region $\mathbb{X}$
for $\mathbb{X} \in \{ \mathbb{L},\mathbb{R},\mathbb{U},\mathbb{D}\}$. 
The points in the  other regions (in white) do not correspond to any intervals. 
Each point has to be considered with the multiplicity when we visualize a given persistence diagram. 
We note that region $\mathbb{U}$ can be thought as a standard persistence diagram. On the other hand, the interval denoted $\langle s, t \rangle$ associated to the point $(s,t)$ in region $\mathbb{D}$
corresponds to a birth $t$ and death $s$ in the lower path.
This notation for $\mathbb{D}$ is opposite to the one for the standard persistence diagrams,
and we explain our reasoning in Remark~\ref{rem:interpret}(1).}
\label{fig:bipathPD}
\end{center}
\end{figure}

\begin{remark}
  \label{rem:interpret}
  We provide the following observations on interpreting bipath persistence diagrams,
  together with some reasoning explaining why we laid out our visualization as we have.
  \begin{enumerate}
  \item In each region, points to the left and/or above correspond to intervals with longer lengths.

    Assuming we fix the ordering for the points of the upper path as $1,2,\hdots,n$,
    this forces us to order the points of the lower path as $m',\hdots,2',1'$ for this to hold.
    Once this ordering is fixed,
    this forces us to denote the intervals $[t,s]$ of the lower path by $\langle s, t \rangle$ as we have,
    to have this hold for region $\mathbb{D}$.
    
  \item Across regions, the following regions can be thought of as ``connected''
    (taken to mean the existence of pairs of intervals with lengths differing by just $1$).
    First, $\mathbb{U}$ and $\mathbb{R}$ (and $\mathbb{R}$ and $\mathbb{D}$) are connected across their shared boundaries
    in Figure~\ref{fig:bipathPD}.
    Furthermore, the left boundary of $\mathbb{U}$ and the right boundary of $\mathbb{L}$ are connected,
    and lower boundary of $\mathbb{L}$ and the upper boundary of $\mathbb{D}$ are connected.
    Finally, the interval $B$ is connected to the upper left corner of $\mathbb{R}$,
    and also to the upper left corner of $\mathbb{L}$.
    
    That is, removing the interval $B$, the bipath persistence diagram naturally lives on a torus,
    with the valid regions forming a cycle
    $ \cdots \relbar \mathbb{U} \relbar \mathbb{R} \relbar \mathbb{D} \relbar \mathbb{L} \relbar \cdots $.
    For example, starting from an interval in $\mathbb{U}$,
    we can progressively lengthen it to the right until it reaches $\hat1$ and extends to the lower path.
    We can then start shortening it from the left until it is in $\mathbb{D}$. Then, lengthening it again, we can make it reach
    $\hat0$ and beyond, giving an interval in $\mathbb{L}$. Shortening it again gives an interval in $\mathbb{U}$; and so on.
  \end{enumerate}
\end{remark}

In Section~\ref{sec:bipathPD}, we will compute the persistence diagram together with its realization for a small example. 

\subsection{Bipath filtrations}
\label{sec:bipath_filtrations}
For a poset $P$, we say that a \emph{$P$-filtration} is a functor $S \colon P \to \mathsf{Top}$ such that $S(a\leq b) \colon S(a) \hookrightarrow S(b)$ is an inclusion map for every $a\leq b$ in $P$, where $\mathsf{Top}$ is the category of topological spaces. 
Applying the $q$th homology functor to $S$ (with coefficient field $k$), we obtain a $P$-persistence module $H_q(S;k)$ if it is pointwise finite (i.e., $H_q(S;k)(a)$ is finite dimension for all $a\in P$). 
That is, 
$H_q(S;k)(a):=H_q(S(a);k)$ is the $q$th homology group of the space $S(a)$, and 
each linear map $H_q(S;k)(a\leq b) \colon H_q(S(a);k) \to H_q(S(b);k)$ is induced by the inclusion map $S(a) \hookrightarrow S(b)$. 

Among others, a $B_{n,m}$-filtration (a bipath filtration) 
can be illustrated as follows:
\begin{equation*}
  S\colon 
  \begin{tikzcd}[row sep=0.1em,column sep = 1.4em, inner sep=0pt]
    & S_1 \rar[hookrightarrow] & S_2 \rar[hookrightarrow] & \cdots \rar[hookrightarrow] & S_n \ar[dr,hookrightarrow] & \\
    S_{\hat0} \ar[ur,hookrightarrow] \ar[dr,hookrightarrow] & & & &  & S_{\hat1} \\
    & S_{1'} \rar[hookrightarrow] & S_{2'} \rar[hookrightarrow] & \cdots \rar[hookrightarrow] & S_{m'} \ar[ur,hookrightarrow] &
  \end{tikzcd}.
\end{equation*}


If $H_q(S;k)$ is pointwise finite, then we obtain $\mathcal{D}_q(S):= \mathcal{D}(H_q(S;k))$ and call it the \emph{$q$th bipath persistence diagram} of $S$ (with coefficient field $k$). 
Notice that it admits a realization $p(\mathcal{D}(S))$ in the plane 
as in Section~\ref{sec:bipath_pd}.

Now, we explain how bipath persistence modules (and hence bipath persistence diagrams) arise from $P$-filtrations, including multiparameter filtrations (the case $P=\mathbb{R}^d$). 

\begin{example}\label{ex:embedding_bipath}
Let $P$ be a poset and $B:=B_{n,m}$ a bipath poset for some $n,m>0$. 
Regarding as a functor, an order embedding $f\colon B\hookrightarrow P$ naturally induces a functor $f^{\ast} \colon \rep(P)\to \rep (B)$ called the \emph{restriction functor}, where $f^{\ast}(V)\in \rep(B)$ is defined by   
\begin{equation*}
    f^{\ast}(V)(a) := V(f(a)) \quad \text{and} \quad 
    f^{\ast}(V)(a\leq b) := V(f(a)\leq f(b))
\end{equation*}
for any $P$-persistence module $V$. 

Let $S\colon P\to \mathsf{Top}$ be a $P$-filtration. 
On one hand, applying the $q$th degree homology functor, one obtains a $P$-persistence module $H_q(S;k)$ whenever it is pointwise finite. 
Then, it gives a bipath persistence module $f^{\ast}(H_q(S;k))$ via the restriction functor. 
On the other hand, we have a bipath filtration $S\circ  f\colon B\to \mathsf{Top}$ as a composition. 
Then, it induces a bipath persistence module $H_q(S\circ f;k)$ by taking homology. 
Since both constructions are functorial, we have an isomorphism 
\begin{equation*}
    f^{\ast}(H_q(S;k)) \cong H_q(S\circ f;k)     
\end{equation*}
in $\rep(B)$. Finally, we obtain its persistence diagram via the interval-decomposition. 

The above construction yields large family of $q$th bipath persistence diagrams starting from a given $P$-filtration $S$ by changing order embeddings.
For more on invariants defined using order embeddings, see \cite{amiot2024invariants}.
\end{example}


\section {Decomposition of bipath persistence modules} 
\label{Section:Dim=0}
In this section, we study the structure of bipath persistence modules. 
As a result, we give a conceptually simple method for the decomposition. 
In particular, we will give a construction of submodules which provide an internal direct sum for it. 
Note that in Section~\ref{section:matrixmethod} we give a potentially more efficient algorithm for computing decompositions using ideas from matrix problems.

Throughout this section, let $V$ be a bipath persistence module over $B:=B_{n,m}$. 
We recall from Proposition~\ref{prop:str_bipathmodule}(2) that $V$ is decomposed as $V=V_B\oplus V_{\mathbb{L}} \oplus V_{\mathbb{R}} \oplus V_{\mathbb{U}} \oplus V_{\mathbb{D}}$. 
Thus, we describe each summand in order to obtain the (interval-)decomposition. 

Our strategy is the following. 

\begin{itemize}
\item In Subsection~\ref{sec:first_step}, 
we will construct submodules $Y,W \subseteq V$ and prove that 
they provide an internal direct sum $V = Y\oplus W$ such that 
$Y=V_B$. 
Furthermore, we compute the multiplicity $m(B)$ of $k_B$ by 
$m(B) = {\rm rank} V(\hat{0}\leq \hat{1})$ in this case.

\item In Subsection~\ref{sec:second_step}, 
we decompose a complement $W$ in the previous step. 
In fact, we will construct submodules $W',W'',Z \subseteq W$ 
and prove that they satisfy $W = W'\oplus W''\oplus Z$, $W' = W_{\mathbb{L}} = V_{\mathbb{L}}$ and 
$W'' = W_{\mathbb{R}} = V_{\mathbb{R}}$. 
Notice that constructions of $W'$ and $W''$ are the dual to each other. 

\item The remaining part $Z$ satisfies $Z(\hat{0}) = Z(\hat{1})=0$. 
Then, we find that the restriction $Z'$ and $Z''$ of $Z$ to the upper path $U$ and the lower path $D$ in \eqref{eq:UandD} coincides with 
$Z_{\mathbb{U}} = V_{\mathbb{U}}$ and $Z_{\mathbb{D}} = V_{\mathbb{D}}$ respectively. 
\end{itemize}

Consequently, we obtain a decomposition 
\begin{equation}\label{eq:VtoLRUD}
    V=V_B\oplus V_{\mathbb{L}} \oplus V_{\mathbb{R}} \oplus V_{\mathbb{U}} \oplus V_{\mathbb{D}} = Y\oplus W'\oplus W''\oplus Z' \oplus Z''.
\end{equation}
It is worth mentioning that each summand $W'$, $W''$, $Z'$, $Z''$ can be regarded as zigzag persistence modules. 
Hence, one can use methods for zigzag persistence modules in order to decompose them.

\subsection{A construction of $V_B$ and its complement}
\label{sec:first_step} 
For a bipath persistence module $V$, we will construct submodules $Y$ and $W$ with the desired property. 

If we have $V(\hat{0} \leq \hat{1})=0$, 
then let $Y := 0$ and $W := V$. 
Otherwise, we define $Y$ and $W$ in the following way. 
By our assumption, we have a non-zero linear map 
\begin{equation*}
V(\hat{0}\leq \hat{1})\colon V(\hat{0}) \to V(\hat{1}). 
\end{equation*}
We regard its image $N := \Image V(\hat{0} \leq \hat{1})$ as a subspace of $V(\hat{0})$ via the retraction $\iota \colon N \to V(\hat{0})$ such that $V(\hat{0}\leq \hat{1}) \circ \iota = 1_{N}$. 
Then, let $Y$ be a submodule of $V$ provided by 
$Y(a):= V(\hat{0}\leq a)(N)$ for all $a\in B$. 
Moreover, we take a complement $\overline{Y}(a)$ of $Y(a)$ in $V(a)$ as  $k$-vector spaces for every $a\in B$. 
Then, let $W$ be a $B$-persistence module defined by $W(a) := \overline{Y}(a)$ and $W(a\leq b) := \pi_b \circ V(a\leq b) \circ \iota_a$, where 
$\iota_a \colon W(a) \to V(a)$ and $\pi_a \colon V(a) \to W(a)$ are canonical inclusions and projections respectively. 
By definition, for every arrow $a\to b$ in $\Hasse(B)$, 
we have the following commutative diagram of vector spaces and linear maps
\begin{equation}\label{eq:repmat1}
\begin{split}
\xymatrix@C=70pt@R=35pt{ 
V(a) 
    \ar[r]^{V(a\leq b)} \ar@{=}[d] & 
V(b) 
    \ar@{=}[d] \\ 
Y(a) \oplus W(a)
    \ar[r]^-{\begin{bsmallmatrix} 
    Y(a\leq b) &  g_{a,b} \\ 0 & W(a\leq b) \\ 
    \end{bsmallmatrix}} 
    \ar@{}[ur]|\circlearrowleft
& 
Y(b) \oplus W(b). 
}
\end{split}
\end{equation}
Thus, we obtain an exact sequence 
\begin{equation}\label{eq:exactseq1}
    0\to Y \to V \to W \to 0    
\end{equation} 
in $\rep(B)$.

Our claim is the following. 

\begin{proposition}\label{prop:basechange1}
    We have an automorphism $\rho \colon V \xrightarrow{\sim} V$ such that $\sigma(Y) = Y$ and $V = Y \oplus \rho(W)$ in $\rep(B)$. 
    In particular, the exact sequence \eqref{eq:exactseq1} splits. 
    Moreover, we have $Y \cong k_B^{{\rm rank} V(\hat{0}\leq \hat{1})}$ and $W(\hat{0}\leq \hat{1})=0$. 
\end{proposition}

Therefore, we may take $Y = V_B$ in the decomposition of $V$. 
The automorphism $\rho$ in the statement can be obtained by using the following elementary result in linear algebra. 

\begin{lemma}[see {\cite[Section~3.5]{Artin91}} for example]
\label{lem:delete_ast}
Suppose that
    \begin{equation}\label{eq:matseq}
U_1\oplus U_2 
\xrightarrow{
\left[
\begin{matrix}
S  &  M \\ 
0 & T \\ 
\end{matrix}
\right]
}
U_1'\oplus U_2' 
\xrightarrow{
\left[
\begin{matrix}
S' & M' \\ 
0 & T' \\ 
\end{matrix}
\right]
}
U_1''\oplus U_2''
\end{equation}
is a sequence of linear maps given by block matrices. 
If the number of rows of the matrix $S'$ equals its rank, then the matrix $M'$ can be reduced to the zero matrix by applying the following operation finitely many times: 
Now, let $u_1,\ldots,u_s$ be a basis of $U_1'$. 
\begin{itemize} 
    \item If $v_1,\ldots,v_t$ is a basis of $U_2'$, then for some $1\leq i \leq s$, $1\leq j \leq t$ and a scalar $\lambda \in k$, replace an element $v_j$ with $v_j + \lambda u_i$ to obtain a new basis of $U_2'$. 
\end{itemize}

Notice that the above operation corresponds to the elementary column (resp., row) operations of the right (resp., left) block matrix in \eqref{eq:matseq}.
In particular, it does not change the matrices $S$, $T$, $S'$, and $T'$. 
\end{lemma}

\begin{proof}[Proof of Proposition~\ref{prop:basechange1}]
    By the definition of $Y$, the $k$-linear map $Y(a\leq b)$ is an isomorphism for every $a,b\in B$. 
    Thus, we clearly have $Y\cong k_B^{{\rm rank} V(\hat{0}\leq \hat{1})}$, 
    where we use ${\rm rank}V(\hat{0}\leq \hat{1}) = {\rm rank}Y(\hat{0}\leq \hat{1})$. 
    
    In the diagram \eqref{eq:repmat1}, if we consider the representation matrices (with fixed bases), we can apply Lemma~\ref{lem:delete_ast} repeatedly for reducing $g_{a,b}$ to $0$ by changing basis of $V(a)$, except at $(a,b) = (\hat{0},1')$. 
    After that, we can deduce $g_{\hat{0},1'}=0$ by the commutativity \eqref{eq:commutativity}. 
    Consequently, they give rise to an automorphism $\rho \colon V\to V$ and a decomposition $V = Y \oplus \rho(W)$ as desired. 
    
    The latter assertion is clear from our construction. 
\end{proof}

\subsection{Constructions of other summands}
\label{sec:second_step}
Let $W$ be a bipath persistence module satisfying $W(\hat{0}\leq \hat{1})=0$ (e.g. we may take $\rho(W)$ for $V$ of the previous subsection).

We define a submodule $W'$ of $W$ by letting $W'(a) := W(\hat{0}\leq a)(W(\hat{0}))$ for all $a\in B$. 
In addition, let $X$ be a $B$-persistence module such that 
$X(a)$ is a complement of $W'(a)$ in $W(a)$ and $X(a\leq b):=\pi'_b \circ W(a\leq b) \circ \iota'_a$, where $\iota_a' \colon X(a) \to W(a)$ and $\pi_a' \colon X(a) \to W(a)$ are the canonical inclusion and projection respectively. 
From our construction, for every arrow $a\to b$ in $\Hasse(B)$, we have a commutative diagram 
\begin{equation}
\begin{split}
\xymatrix@C=70pt@R=35pt{ 
W(a) 
    \ar[r]^{V(a\leq b)} \ar@{=}[d] & 
W(b) 
    \ar@{=}[d] \\ 
W'(a) \oplus X(a)
    \ar[r]^-{\begin{bsmallmatrix} 
W'(a \leq b) &  g_{a,b}' \\ 0 & X(a\leq b) \\ 
    \end{bsmallmatrix}} 
    \ar@{}[ur]|\circlearrowleft
& 
W'(b) \oplus X(b) 
}
\end{split}
\end{equation}
and an exact sequence \begin{equation}\label{eq:exactseq2}
    0\to W' \to W \to X \to 0 
\end{equation} 
in $\rep(B)$. 

We have the following result.  

\begin{proposition}\label{prop:basechange2}
    We have an automorphism $\sigma \colon W \xrightarrow{\sim} W$ such that $\sigma(W') = W'$ and $W = W'\oplus \sigma(X)$ in $\rep(B)$. 
    Thus, the exact sequence \eqref{eq:exactseq2} splits.
    Moreover, they satisfy $W'(\hat{1})=0$ and $X(\hat{0})=0$. 
\end{proposition}

\begin{proof}
    By the definition of $W$, the map $W'(a\leq b) \colon W'(a)\to W'(b)$  is surjective for every $a\to b$ in $\Hasse(B)$. 
    Then, we can use Lemma~\ref{lem:delete_ast} and get the former assertion by a similar discussion to a proof of Proposition \ref{prop:basechange1}. The latter assertion is obvious from the former one. 
\end{proof}

In the above situation, we have $W' = W_{\mathbb{L}}$. 
Moreover, we can apply the dual construction to $\sigma(X)$ to obtain a decomposition $W = W' \oplus \sigma(X) = W' \oplus W'' \oplus Z$ such that $W'' = W_{\mathbb{R}}$ and $Z(\hat{0}) = Z(\hat{1})=0$ by the dual statement of Proposition~\ref{prop:basechange2}. 
For $Z$, the restrictions $Z'$ and $Z''$ to the upper path $U$ and the lower path $D$ give a decomposition of $Z$, that is, $Z=Z'\oplus Z''$ satisfying $Z'=Z_{\mathbb{U}}$ and $Z''=Z_{\mathbb{D}}$ respectively.

Now, we go back to the decomposition of $V$. 
As a summary of this section, we have the following.

\begin{proposition}
    Let $Y, W',W'',Z',Z''$ be bipath persistence modules defined for $V$ in Subsections 
    ~\ref{sec:first_step} and \ref{sec:second_step}. 
    Then, they give the desired decomposition of $V$ in \eqref{eq:VtoLRUD}. 
\end{proposition}

\begin{proof}
    All modules in the statement can be taken as submodules of $V$ satisfying the desired properties. 
\end{proof}

\section{Main algorithm: Matrix Method }\label{section:matrixmethod}

\newcommand{\upperp}{U}
\newcommand{\lowerp}{D}
\newcommand{\upperX}{X}
\newcommand{\lowerX}{Y}
\newcommand{\lefts}{\ell}
\newcommand{\rights}{r}
\newcommand{\centers}{c}
\newcommand{\others}{o}

In this section, we use the idea of matrix problems methods
(see for example \cite[Chapter~1]{gabriel1992representations}, \cite{asashiba2019matrix})
to provide an algorithm for computing
interval decompositions of bipath persistence modules,
independent of the decomposition given in Section~\ref{Section:Dim=0}. 
Note that matrix problems can be understood within the theoretical framework
of the representation theory of bocses \cite{crawley1988tame,rojter2006matrix}
or modules over ditalgebras \cite{bautista2009differential},
but to keep the necessary theoretical background to a minimum
we do not use these formulations.

We first give an overview of the idea for this algorithm.
We fix a bipath poset $B:=B_{n,m}$ and a bipath persistence module $V$ over $B$. 
Recall that $V$ is described by a diagram 
\[    
    \begin{tikzcd}[row sep=0.1em, column sep=4.5em, inner sep=0pt]
      & V(1) \rar{V(1\leq 2)} & V(2) \rar{V(2\leq 3)} & \cdots \rar{V(n-1\leq n)} & V(n) \ar[dr]{}{V(n\leq \hat{1})} & \\ 
      V(\hat{0}) \ar[ur]{}{V(\hat{0} \leq 1)} 
      \ar[dr]{}{V(\hat{0}\leq 1')} & & & & & V(\hat{1}). \\    
      & V(1') \rar{V(1'\leq 2')} & V(2') \rar{V(2'\leq 3')} & \cdots \rar{V(m-1' \leq m')} & V(m') \ar[ur]{}{V(m'\leq \hat{1})} &
    \end{tikzcd}
\]

For the upper path $\upperp$ and the lower path $\lowerp$ in \eqref{eq:UandD}, we restrict $V$ to obtain persistence modules $V_\upperp$ and $V_\lowerp$ over $\upperp$ and $\lowerp$ respectively. 
\begin{equation*}
    \begin{tikzcd}[row sep=0.5em, column sep=4.5em, inner sep=0pt]
    V_U \colon \ V(\hat{0}) \ar[r]{}{V(\hat{0} \leq 1)} & V(1) \rar{V(1\leq 2)} & V(2) \rar{V(2\leq 3)} & \cdots \rar{V(n-1\leq n)} & V(n) \ar[r]{}{V(n\leq \hat{1})} &  V(\hat{1}), \\ 
    V_D \colon \ V(\hat{0}) \ar[r]{}{V(\hat{0} \leq 1')} & V(1') \rar{V(1'\leq 2')} & V(2') \rar{V(2'\leq 3')} & \cdots \rar{V(m-1' \leq m')} & V(m') \ar[r]{}{V(m'\leq \hat{1})} &  V(\hat{1}).
    \end{tikzcd}
\end{equation*}

Since $U$ and $D$ are the equioriented $A$-type posets, both are interval-decomposable as in Section \ref{sec:bipath_modules}. 
That is, there exist isomorphisms of persistence modules 

\begin{eqnarray*}
  \begin{tikzcd}
    V_\upperp \rar{\phi} & \bigoplus\limits_{I \in \Int{\upperp}} k_{I}^{m_\upperp(I)} =: \upperX
  \end{tikzcd}
  & (\text{in } \rep(\upperp)), \\
  \begin{tikzcd}
    V_\lowerp \rar{\psi} & \bigoplus\limits_{I \in \Int{\lowerp}} k_{I}^{m_\lowerp(I)} =: \lowerX
  \end{tikzcd}
  & (\text{in } \rep(\lowerp)),
\end{eqnarray*}
which gives interval decompositions for $V_\upperp$ and $V_\lowerp$ independently of each other.
However, $V_\upperp$ and $V_\lowerp$ share the same vector space at $\hat0$ and $\hat1$,
and we may not be able to combine $\phi$ and $\psi$ as-is in order to obtain
an isomorphism of $V$ with an interval-decomposable representation of $\rep(B)$. 

Thus, we consider two change-of-basis matrices:
one (which we call $\Lambda$ below) at $\hat0$ for the vector space $V_\upperp(\hat0) = V_\lowerp(\hat0)$
connecting the basis specified by $\phi_{\hat0}$ and the basis specified by $\psi_{\hat0}$
and another one (which we call $\Gamma$ below) at $\hat1$ for the vector space
$V_\upperp(\hat1) = V_\lowerp(\hat1)$ defined similarly.
Using elementary matrix operations we transform them in a way (coinciding with changes of bases)
to get bases for all the vector spaces of $V$ coinciding with an interval-decomposition.
While $\Lambda$ and $\Gamma$ are defined as matrices,
we declare certain elementary operations as not permissible,
and perform only the permissible elementary operations. This is to ensure that
the interval decompositions already obtained for the upper and lower paths
are not destroyed.

That is, the above defines a matrix problem $M(\phi, \psi)$
which is given as a block matrix together with certain permissible operations.
Solving this matrix problem (i.e.\ transforming it to a standard form) corresponds with
computing an interval decomposition of $V$.
This ends the high-level overview of the algorithm.

\subsection{Defining the block matrix problem} 
Keeping the above notation, we consider the isomorphisms of vector spaces
$\Lambda := \psi_{\hat0} \phi_{\hat0}^{-1}$ at $\hat0$
and
$\Gamma := \psi_{\hat1} \phi_{\hat1}^{-1}$ at $\hat1$,
which makes the diagrams
\[
  \begin{tikzcd}
    V_\upperp(\hat0) \ar[d,equal] \rar{\phi_{\hat0}} & \upperX(\hat0) \dar{\Lambda := \psi_{\hat0} \phi_{\hat0}^{-1}}\\
    V_\lowerp(\hat0) \rar{\psi_{\hat0}} & \lowerX(\hat0)
  \end{tikzcd}
  \quad \text{ and } \quad 
  \begin{tikzcd}
    V_\upperp(\hat1) \ar[d,equal] \rar{\phi_{\hat1}} & \upperX(\hat1) \dar{\Gamma := \psi_{\hat1} \phi_{\hat1}^{-1}}\\
    V_\lowerp(\hat1) \rar{\psi_{\hat1}} & \lowerX(\hat1)
  \end{tikzcd}
\]
commute.
These form the front and back faces of the following cube:
\begin{equation}
  \label{eq:bigcube}
  \begin{tikzcd}
    V_\upperp(\hat0) \ar[dd,equal] \ar[rr,pos=0.2]{}{\phi_{\hat0}} && \upperX(\hat0) \ar[dd]{}{\Lambda} \ar[drrr,sloped,pos=0.2]{}{\upperX(\hat0\leq\hat1)}  \\
    && & V_\upperp(\hat1) \ar[rr,pos=0.2]{}{\phi_{\hat1}} \ar[ulll,swap,sloped,pos=0.8,crossing over,<-]{}{V_\upperp(\hat0\leq\hat1)}  && \upperX(\hat1) \ar[dd]{}{\Gamma}\\
    V_\lowerp(\hat0) \ar[rr,pos=0.2]{}{\psi_{\hat0}} \ar[drrr,swap,sloped,pos=0.2]{}{V_\lowerp(\hat0\leq\hat1)}  && \lowerX(\hat0) \ar[drrr,sloped,pos=0.2]{}{\lowerX(\hat0\leq\hat1)} \\
    && & V_\lowerp(\hat1) \ar[uu,equal,crossing over] \ar[rr,pos=0.2]{}{\psi_{\hat1}} && \lowerX(\hat1)
  \end{tikzcd}
\end{equation}
The left face commutes because $V_\upperp(\hat0\leq\hat1) = V(\hat0\leq\hat1) = V_\lowerp(\hat0\leq\hat1)$.
The top and bottom faces commutes because $\phi$ and $\psi$ are morphisms of representations.
Thus, the right face
\begin{equation}
  \label{eq:leftrightlink}
  \begin{tikzcd}[column sep=3em,]
    \upperX(\hat0) \rar{\upperX(\hat0\leq\hat1)} \dar{\Lambda} & \upperX(\hat1) \dar{\Gamma}\\
    \lowerX(\hat0) \rar{\lowerX(\hat0\leq\hat1)} & \lowerX(\hat1)
  \end{tikzcd}
\end{equation}
also commutes.

As we will need it later, and to give some theoretical framework for our manipulations,
we define the poset

\begin{equation}
  \label{diag:extendedbipath}
  B' \colon
  \begin{tikzcd}[inner sep=0pt]
    \hat{0} \dar \rar & 1 \rar & 2 \rar & \cdots \rar & n \rar  & \hat{1}' \dar \\
    \hat{0}' \rar & 1' \rar & 2' \rar & \cdots \rar & m' \rar  &  \hat{1}
  \end{tikzcd},
\end{equation}
which contains $B$ as a subposet,
and consider the inclusion functor
\[
  \iota \colon \rep (B) \hookrightarrow \rep (B')
\]
where for $V \in \rep (B)$, we set
\begin{align*}
  \iota(V)(a) & := V(a) \text{ for } a\in B, \\
  \iota(V)(x') & := V(x) \text{ for } x \in \{\hat{0},\hat{1}\}, \\  
  \iota(V)(\hat0 \leq \hat0') & := 1_{V(\hat0)},\\
  \iota(V)(\hat1' \leq \hat1) & := 1_{V(\hat1)},\\
\end{align*}
and a similar definition for morphisms $f \colon V\rightarrow W$ in $\rep (B)$.

  \begin{lemma}
    \label{lem:formalism0}
    The functor $\iota$ as defined above is a fully faithful $k$-linear functor.
  \end{lemma}
  Then, the right face of Diagram~\eqref{eq:bigcube} is simply the
  representation $\iota(V)$ restricted to the corners $\hat0, \hat0', \hat1, \hat1'$ of $B'$,
  and the above argument shows the following.

\begin{lemma}
  \label{lem:formalism1}
  The isomorphisms $\phi$ and $\psi$
  induce an isomorphism $f \colon \iota(V) \rightarrow Z$ in $\rep (B')$, 
  where $Z$ is the $B'$-persistence module given by 
  \[
  \begin{tikzcd}[inner sep=0pt]
    \upperX(\hat{0}) \dar[swap]{\Lambda} \rar  & \upperX(1) \rar  & \cdots \rar & \upperX(n) \rar  & \upperX(\hat{1}) \dar{\Gamma} \\
    \lowerX(\hat{0}') \rar  & \lowerX(1') \rar  & \cdots \rar & \lowerX(m') \rar  &  \lowerX(\hat{1}').
  \end{tikzcd}
  \]
\end{lemma}

Note that the representation $Z$ (restricted to the corners $\hat0, \hat0', \hat1, \hat1'$) is the right face of Diagram~\eqref{eq:bigcube},
and the proof is essentially the commutativity of Diagram~\eqref{eq:bigcube}.

In the rest of this subsection until the start of subsection~\ref{subsec:normal_form_to_decompo},
we will define a block matrix problem and provide an algorithm for computing its normal form.
While we will not directly refer to the above formalization,
the block matrix problem and computing its normal form corresponds
to finding an isomorphism of $\iota(V)$ with a representation $Z$ in $\rep (B')$
with $\Lambda$ and $\Gamma$ essentially identity maps (in fact, permutation maps).
From such a form we can read out an interval decomposition of $\iota(V)$, and thus of $V$.
This connection will be explained in Subsection~\ref{subsec:normal_form_to_decompo}.

We partition the direct summands of $\upperX := \bigoplus\limits_{I \in \Int{\upperp}} k_{I}^{m(I)}$
as $\upperX = \upperX_\lefts \oplus \upperX_\centers \oplus \upperX_\rights \oplus \upperX_\others$
where $\upperX_\centers$ consists of all terms that are nonzero at $\hat0$ and $\hat1$,
$\upperX_\lefts$ consists all terms that are nonzero at $\hat0$ and zero at $\hat1$,
$\upperX_\rights$ consists all terms that are zero at $\hat0$ and nonzero at $\hat1$,
and
$\upperX_\others$ are the remaining terms.
We also define the decomposition
$\lowerX = \lowerX_\lefts \oplus \lowerX_\centers \oplus \lowerX_\rights \oplus \lowerX_\others$
similarly.
The intervals in $\upperX_\others$ and $\lowerX_\others$ do not intersect $\hat0$ nor $\hat1$,
and thus are themselves also (isomorphic to) intervals for the original bipath persistence module $V$.

Thus, commutative Diagram~\eqref{eq:leftrightlink} becomes
\begin{equation}
  \label{eq:leftrightlink2}
  \begin{tikzcd}[ampersand replacement=\&, column sep=3.5em,]
    \upperX_\lefts(\hat0) \oplus \upperX_\centers(\hat0)
    \rar{\smat{0&E\\0&0}}
    \dar{\Lambda}
    \&
    \upperX_\centers(\hat1) \oplus \upperX_\rights(\hat1) \dar{\Gamma}\\
    \lowerX_\lefts(\hat0) \oplus \lowerX_\centers(\hat0)
    \rar{\smat{0&E\\0&0}}
    \&
    \lowerX_\centers(\hat1) \oplus \lowerX_\rights(\hat1)
  \end{tikzcd}
\end{equation}
where
\begin{itemize}
\item we identify the linear maps
  with their matrix representations
  relative to the standard basis given by the decomposition of
  $\upperX(\hat0)$,
  $\upperX(\hat1)$,
  $\lowerX(\hat0)$, and
  $\lowerX(\hat1)$,
\item $0$ and $E$ are block matrices of appropriate sizes, and
\item $\upperX(\hat0\leq\hat1)$ and $\lowerX(\hat0\leq\hat1)$ take on the given forms
  due to the definition of
  $\upperX_\lefts, \upperX_\centers,\upperX_\rights,\upperX_\others$ and of
  $\lowerX_\lefts, \lowerX_\centers,\lowerX_\rights,\lowerX_\others$.
\end{itemize}
Then, writing $\Lambda$ and $\Gamma$ in block matrix form
\begin{eqnarray*}
  \Lambda =
  \begin{bmatrix}
    \pi_\lefts \Lambda \iota_\lefts & \pi_\lefts \Lambda \iota_\centers \\
    \pi_\centers \Lambda \iota_\lefts & \pi_\centers \Lambda \iota_\centers
  \end{bmatrix}\colon
  \upperX_\lefts(\hat0) \oplus \upperX_\centers(\hat0)
  \rightarrow
  \lowerX_\lefts(\hat0) \oplus \lowerX_\centers(\hat0) \\
  \Gamma =
  \begin{bmatrix}
    \pi_\centers \Gamma \iota_\centers & \pi_\centers \Gamma \iota_\rights \\
    \pi_\rights \Gamma \iota_\centers & \pi_\rights \Gamma \iota_\rights
  \end{bmatrix}\colon
  \upperX_\centers(\hat1) \oplus \upperX_\rights(\hat1)
  \rightarrow
  \lowerX_\centers(\hat1) \oplus \lowerX_\rights(\hat1)
\end{eqnarray*}
the commutativity of Diagram~\eqref{eq:leftrightlink2}
implies that
\begin{equation}
  \label{eq:zero}
  \pi_\centers \Lambda \iota_\lefts = 0,
  \pi_\rights \Gamma \iota_\centers = 0,
  \text{ and }
  \pi_\centers \Lambda \iota_\centers  = \pi_\centers \Gamma \iota_\centers
\end{equation}
where $\pi$ and $\iota$ indicate the appropriate projections and inclusions.
That is, we have block diagonal forms
\begin{eqnarray}
  \label{eq:lambdablock}
  \Lambda =
  \begin{bmatrix}
    \pi_\lefts \Lambda \iota_\lefts & \pi_\lefts \Lambda \iota_\centers \\
    0 & \pi_\centers \Lambda \iota_\centers
  \end{bmatrix}:
  \upperX_\lefts(\hat0) \oplus \upperX_\centers(\hat0)
  \rightarrow
  \lowerX_\lefts(\hat0) \oplus \lowerX_\centers(\hat0) \\
  \label{eq:gammablock}
  \Gamma =
  \begin{bmatrix}
    \pi_\centers \Gamma \iota_\centers & \pi_\centers \Gamma \iota_\rights \\
    0 & \pi_\rights \Gamma \iota_\rights
  \end{bmatrix}:
  \upperX_\centers(\hat1) \oplus \upperX_\rights(\hat1)
  \rightarrow
  \lowerX_\centers(\hat1) \oplus \lowerX_\rights(\hat1)
\end{eqnarray}

We will need the following property later.
\begin{lemma}
  \label{lem:invertible}
  The blocks
  $\pi_\lefts \Lambda \iota_\lefts$,
  $\pi_\centers \Lambda \iota_\centers  = \pi_\centers \Gamma \iota_\centers$,
  and
  $\pi_\rights \Gamma \iota_\rights$
  in Equations~\eqref{eq:lambdablock},\ \eqref{eq:gammablock} are invertible.
\end{lemma}
\begin{proof}
  We have
  \[
    \left(\dim \upperX_\centers(\hat0) = \dim \upperX_\centers(\hat1)\right) = \dim\Image \upperX(\hat0 \leq \hat1)
    =
    \dim\Image \lowerX(\hat0 \leq \hat1) = \left(\dim \lowerX_\centers(\hat0) = \dim \lowerX_\centers(\hat1)\right),
  \]
  and thus the blocks in question are square matrices. Furthermore,
  the block diagonal form implies that
  $\det\Lambda = \det(\pi_\lefts \Lambda \iota_\lefts) \det(\pi_\centers \Lambda \iota_\centers)$
  and
  $\det\Gamma = \det(\pi_\centers \Gamma \iota_\centers)\det(\pi_\rights \Gamma \iota_\rights)$.
  The result follows from the invertibility of $\Lambda$ and $\Gamma$.
\end{proof}

We gather the two block matrices of Equations~\eqref{eq:lambdablock},\ \eqref{eq:gammablock}
into a single block matrix
\begin{equation}
  \label{eq:blockmatrix}
  \begin{bmatrix}
    \pi_\lefts \Lambda \iota_\lefts & \pi_\lefts \Lambda \iota_\centers & \\
    0  & \pi_\centers \Lambda \iota_\centers = \pi_\centers \Gamma \iota_\centers & \pi_\centers \Gamma \iota_\rights \\
    & 0 & \pi_\rights \Gamma \iota_\rights
  \end{bmatrix}
\end{equation}
where the upper-right block and the lower-left block are not defined.
This can be interpreted as a notational convenience,
or alternatively can be thought of as a block matrix with some blocks left always empty or ``strongly zero''.
When we perform matrix row or column operations on such block matrices,
strongly zero blocks remain zero no matter what operations are done to them,
in contrast to zero blocks that may become nonzero.
To simplify the notation, and denote the possible block statuses as we perform matrix operations,
we adapt the notation of \cite{asashiba2019matrix}.
\begin{notation}
  We write
  \begin{itemize}
  \item $*$ for unprocessed blocks,
  \item $\n$ for strongly zero blocks,
  \item $E$ for identity blocks (of appropriate sizes),
    and
  \item $0$ for zero blocks (of appropriate sizes).
  \end{itemize}
  The blocks marked as $\n$, $E$, and $0$ are considered processed.
\end{notation}
The single block matrix will then be expressed as
\begin{equation}
  \label{eq:blockmatrix2}
  \mattikz{
    \matheader{
      * \& * \& \n \\
      0 \& * \& *  \\
      \n \& 0 \& *  \\
    };
  }.
\end{equation}

Next, let us consider the permissible matrix operations that preserve the interval-decomposed form
of the upper and lower parts.
We reuse the permissible operations discussed in \cite[{Subsection~3.3 and Definition~6}]{asashiba2019matrix}.
Before going into the details, we note the following differences with \cite{asashiba2019matrix}.
\begin{enumerate}
\item As opposed to \cite{asashiba2019matrix} which considers block matrices consisting of morphisms between
  representations of an $A$-type poset, $\Lambda$ and $\Gamma$ are simply linear maps between vector spaces.
  However, since $\Lambda$ and $\Gamma$ are attached to representations of $A$-type poset
  ($\upperp$ and $\lowerp$, at $\hat0$ and $\hat1$), we use only the operations permitted by this additional structure.
\item In contrast to \cite{asashiba2019matrix} which considers all intervals,
  $\Lambda$ and $\Gamma$ are attached at $\hat0$ and $\hat1$ respectively,
  and thus we only need to consider the structure of intervals of the form $k_{[\hat0, ?]}$ and $k_{[?, \hat1]}$. These are the injective and projective indecomposable representations of an $A$-type poset.
\end{enumerate}

Let $Q$ be either $\upperp$ or $\lowerp$ (an $A$-type poset).
To treat both cases in parallel to facilitate the explanation, we write
\[
  Q:
  \begin{tikzcd}[row sep=0.1em, outer sep=0pt]
    \hat0 =: \Qf{1} \rar & \Qf{2} \rar & \cdots \rar & \Qf{n}-\Qf{1} \rar & \Qf{n} := \hat1
  \end{tikzcd}
\]
where $\Qf{n}$ is $n+2$ for $\upperp$ and $m+2$ for $\lowerp$.
That is, $\Qf{1} := \hat0$, $\Qf{n} := \hat1$, and the intermediate numbers represent the points in-between.

Then, it is well known
that $\dim\Hom(k_I,k_J)$ is either $0$ or $1$ for intervals $I,J \in \Int{Q}$ (see for example \cite[Lemma~1]{asashiba2019matrix}).
Following \cite{asashiba2019matrix} we write
\[
  I \reltoeq J \text{ if and only if } \dim\Hom(k_I,k_J) \neq 0
\]
and write $I \relto J$ whenever $I \reltoeq J$ and $I \neq J$.
We also fix the following morphism
$f\itoi{a,b,c,d} \colon  k_{\Intp{a,b}} \rightarrow k_{\Intp{c,d}}$ defined by
\begin{equation*}
  \left(f\itoi{a,b,c,d}\right)_{\Qf{i}} =
  \begin{cases}
    1_k & \text{for } \Qf{i} \in \Intp{a,b} \cap \Intp{c,d}, \\
    0 & \text{otherwise}
  \end{cases}
\end{equation*}
for each pair of intervals $\Intp{a,b}$,
$\Intp{c,d}$ with $\Intp{a,b} \reltoeq \Intp{c,d}$.
If $\Intp{a,b} \reltoeq \Intp{c,d}$,
$\Intp{c,d} \reltoeq \Intp{e,f}$, \emph{and} $\Intp{a,b} \reltoeq \Intp{e,f}$, then
it is immediate that
\[
  f\itoi{a,b,e,f} = f\itoi{c,d,e,f} f\itoi{a,b,c,d}.
\]
Note that $\reltoeq$ defines a reflexive and antisymmetric (but in general non-transitive) relation on $\Int{Q}$.

Restricted to the intervals $[?, \Qf{n}]$ corresponding to the projective representations
and the intervals $[\Qf{1}, ?]$ corresponding to the injective representations, we have
\[
  [\Qf{n}, \Qf{n}] \relto [\Qf{n}-\Qf{1}, \Qf{n}] \relto \cdots \relto [\Qf{2}, \Qf{n}]
  \relto [\Qf{1}, \Qf{n}]
  \relto [\Qf{1}, \Qf{n}-\Qf{1}] \relto \cdots \relto [\Qf{1}, \Qf{2}] \relto [\Qf{1}, \Qf{1}].
\]
or
\begin{equation}
  \label{eq:order}
  \Intp{1,1}
  \oprelto
  [\Qf{1}, \Qf{2}]
  \oprelto
  \cdots
  \oprelto
  [\Qf{1}, \Qf{n}-\Qf{1}]
  \oprelto
  [\Qf{1}, \Qf{n}]
  \oprelto
  [\Qf{2}, \Qf{n}]
  \oprelto
  \cdots
  \oprelto
  [\Qf{n}-\Qf{1}, \Qf{n}]
  \oprelto
  [\Qf{n}, \Qf{n}]
  .
\end{equation}
Due to the lack of transitivity, this is not a total order. However, we can extend it to a total order.
Then, to be able to easily describe the permissible operations for our matrix problem, we now impose
an order on the bases for the domains (columns) and codomains (rows) of $\Lambda$ and $\Gamma$ by following the above total order.

In more detail (in the new notation for points of the poset $Q = \upperp \text{ or } \lowerp$),
recall that $\Lambda \colon \upperX(\Qf{1}) \rightarrow \lowerX(\Qf{1})$
where
$\upperX(\Qf{1}) = \upperX_\lefts(\Qf{1}) \oplus \upperX_\centers(\Qf{1})$
with $\upperX_\lefts$ a direct sum of interval modules $k_{[\Qf{1}, ?]}$ but not equal to $k_{[\Qf{1}, \Qf{n}]}$,
and $\upperX_\centers$ a direct sum of interval modules $k_{[\Qf{1}, \Qf{n}]}$.
Now, $\upperX_\lefts(\Qf{1})$ is just a direct sum of copies of the vector space $k$, each
coming from the vector space $k$ at point $\Qf{1}$ of some representation $k_{[\Qf{1}, ?]}$.
We order the standard basis by the order~\eqref{eq:order} on the corresponding intervals $[\Qf{1}, ?]$ from whence they came.
Note that each interval $[\Qf{1}, ?]$ can appear with multiplicity in $\upperX_\lefts$.
We do not impose any particular order within basis elements corresponding to the same interval.
The order for the basis of the other vector spaces coming from $\upperX$ (and $\lowerX$) are defined similarly (i.e.~with reference to the order~\eqref{eq:order}).

We can decorate the rows and columns of the block matrix~\eqref{eq:blockmatrix}(its abbreviated form \eqref{eq:blockmatrix2})
to indicate the original intervals which they correspond to. Note once again that each
interval can appear multiple times, and thus each interval denotes multiple rows (or columns).
\begin{equation}
  \label{eq:labelledblockmatrix}
  \mattikz{
    \matheaderwide{
      *           \& \cdots \& *      \&                * \& \phantom{*} \& \phantom{*} \& \phantom{*} \\
      \vdots      \& \ddots \& \vdots \& \vdots           \&             \&             \&  \\
      *           \& \cdots \& *      \&                * \&             \&             \&  \\
      0           \& \cdots \& 0      \& *                \& *           \& \cdots      \& * \\
      \phantom{*} \&        \&        \& 0                \& *           \& \cdots      \& * \\
      \phantom{*} \&        \&        \& \vdots           \& \vdots      \& \ddots      \& \vdots \\
      \phantom{*} \&        \&        \& 0                \& *           \& \cdots      \& * \\
    };
    \node[anchor = south east] (p-0-0) at (p-2-1.west |- p-1-1.north west) {};
    \foreach[count = \i] \v in {
      \({[\Qf1,\Qf1]}\),
      \({\vphantom{\Qf{1}}\smash{\vdots}\;\;}\),
      \({[\Qf{1},\Qf{n}{-}\Qf{1}]}\),
      \({[\Qf{1},\Qf{n}]}\),
      \({[\Qf{2},\Qf{n}]}\),
      \({\vphantom{\Qf{1}}\smash{\vdots}\;\;}\),
      \({[\Qf{n},\Qf{n}]}\)
      }
    {
      \node[left=3.5ex] at (p-\i-1) {\v};
    }
    \foreach[count = \i] \v in {
      \({[\Qf1,\Qf1]}\),
      \({\cdots}\),
      \({[\Qf{1},\Qf{n}{-}\Qf{1}]}\),
      \({[\Qf{1},\Qf{n}]}\),
      \({[\Qf{2},\Qf{n}]}\),
      \({\cdots}\),
      \({[\Qf{n},\Qf{n}]}\)
      }
      {
       \node[above=1.5ex] at (p-1-\i) {\v};
     }
  }.
\end{equation}

\begin{definition}
  \label{defn:permissible}
Following \cite[Definition~6]{asashiba2019matrix}, we declare as \emph{permissible} the following elementary row or column operations on the block matrix~\eqref{eq:blockmatrix2}.
\begin{enumerate}
\item All elementary operations within rows (or columns) with the same label are permissible.
\item If $[\Qf{a},\Qf{b}] \oprelto [\Qf{c},\Qf{d}]$
  then additions of multiples of columns labelled $[\Qf{a},\Qf{b}]$ to columns labelled $[\Qf{c},\Qf{d}]$ are permissible.
\item (Dually,) if $[\Qf{a},\Qf{b}] \oprelto [\Qf{c},\Qf{d}]$
  then additions of multiples from rows labelled $[\Qf{c},\Qf{d}]$ to rows labelled $[\Qf{a},\Qf{b}]$ are permissible.
\end{enumerate}
\end{definition}

Let us recap the above considerations as one Definition.
\begin{definition}
  \label{defn:matrixproblem}
  Let $V \in \rep(B)$ and
  $\phi \colon V_\upperp \rightarrow  \bigoplus_{I \in \Int{\upperp}} k_{I}^{m_\upperp(I)}$,
  $\psi  \colon V_\lowerp \rightarrow \bigoplus_{I \in \Int{\lowerp}} k_{I}^{m_\lowerp(I)}$
  be isomorphisms to interval decompositions of the upper and lower parts of $V$.
  We call the block matrix~\eqref{eq:labelledblockmatrix} (the block matrix~\eqref{eq:blockmatrix} with the order of its rows and columns imposed by $\oprelto$), together with the the permissible operations in Definition~\ref{defn:permissible},
  a \emph{block matrix problem associated to the pair $(\phi,\psi)$}.
\end{definition}
Note that in the above definition, there can be multiple different block matrix problems associated to the same $(\phi, \psi)$. These arise because we have not fixed any particular order among the rows (and columns) with the same interval labels. Up to permutations of rows and columns with the same label, there is essentially only one block matrix problem associated to $(\phi, \psi)$. We denote this by $M(\phi,\psi)$, allowing for some ambiguity on the order of the rows and columns with the same labels.

\subsection{Obtaining a normal form}
\label{subsec:normalform}
Next, we perform permissible row and column operations on the block matrix of $M(\phi,\psi)$
to transform it to a normal form. 
We show that this normal form is a block matrix problem of a pair $(\phi',\psi')$,
and that we can then combine the isomorphisms $\phi'$ and $\psi'$ to get an interval decomposition of $V$.

The block with rows and columns labelled $[\Qf{1},\Qf{n}]$ (the central block) is nothing but
$\pi_\centers \Lambda \iota_\centers  = \pi_\centers \Gamma \iota_\centers$, which was shown to be invertible in
Lemma~\ref{lem:invertible}. Using elementary row and column operations within the rows labelled $[\Qf{1},\Qf{n}]$ and the columns labelled $[\Qf{1},\Qf{n}]$, we can transform it to an identity matrix $E$, giving the following form for the block matrix.
\[
  \mattikz{
    \matheaderwide{
      *           \& \cdots \& *      \&                * \& \phantom{*} \& \phantom{*} \& \phantom{*} \\
      \vdots      \& \ddots \& \vdots \& \vdots           \&             \&             \&  \\
      *           \& \cdots \& *      \&                * \&             \&             \&  \\
      0           \& \cdots \& 0      \& E                \& *           \& \cdots      \& * \\
      \phantom{*} \&        \&        \& 0                \& *           \& \cdots      \& * \\
      \phantom{*} \&        \&        \& \vdots           \& \vdots      \& \ddots      \& \vdots \\
      \phantom{*} \&        \&        \& 0                \& *           \& \cdots      \& * \\
    };
    \node[anchor = south east] (p-0-0) at (p-2-1.west |- p-1-1.north west) {};
    \foreach[count = \i] \v in {
      \({[\Qf1,\Qf1]}\),
      \({\vphantom{\Qf{1}}\smash{\vdots}\;\;}\),
      \({[\Qf{1},\Qf{n}{-}\Qf{1}]}\),
      \({[\Qf{1},\Qf{n}]}\),
      \({[\Qf{2},\Qf{n}]}\),
      \({\vphantom{\Qf{1}}\smash{\vdots}\;\;}\),
      \({[\Qf{n},\Qf{n}]}\)
    }
    {
      \node[left=3.5ex] at (p-\i-1) {\v};
    }
    \foreach[count = \i] \v in {
      \({[\Qf1,\Qf1]}\),
      \({\cdots}\),
      \({[\Qf{1},\Qf{n}{-}\Qf{1}]}\),
      \({[\Qf{1},\Qf{n}]}\),
      \({[\Qf{2},\Qf{n}]}\),
      \({\cdots}\),
      \({[\Qf{n},\Qf{n}]}\)
    }
    {
      \node[above=1.5ex] at (p-1-\i) {\v};
    }
  }.
\]

Then, using elementary column operations adding appropriate multiples
of the columns labelled $[\Qf{1},\Qf{n}]$ to the columns labelled
$[\Qf{2},\Qf{n}], \hdots, [\Qf{n},\Qf{n}]$,
we can zero out all the blocks to the right of the central $E$.
Similarly, using elementary row operations adding appropriate multiples
of the rows labelled $[\Qf{1},\Qf{n}]$ to the rows labelled
$[\Qf{1},\Qf{n}{-}\Qf{1}], \hdots, [\Qf{1},\Qf{1}]$,
we can zero out all the blocks above the central $E$.
Note that these are all permissible operations.
The block matrix becomes
\[
  \mattikz{
    \matheaderwide{
      *           \& \cdots \& *      \& 0      \& \phantom{*} \& \phantom{*} \& \phantom{*} \\
      \vdots      \& \ddots \& \vdots \& \vdots \&             \&             \&  \\
      *           \& \cdots \& *      \& 0      \&             \&             \&  \\
      0           \& \cdots \& 0      \& E      \& 0           \& \cdots      \& 0 \\
      \phantom{*} \&        \&        \& 0      \& *           \& \cdots      \& * \\
      \phantom{*} \&        \&        \& \vdots \& \vdots      \& \ddots      \& \vdots \\
      \phantom{*} \&        \&        \& 0      \& *           \& \cdots      \& * \\
    };
    \node[anchor = south east] (p-0-0) at (p-2-1.west |- p-1-1.north west) {};
    \foreach[count = \i] \v in {
      \({[\Qf1,\Qf1]}\),
      \({\vphantom{\Qf{1}}\smash{\vdots}\;\;}\),
      \({[\Qf{1},\Qf{n}{-}\Qf{1}]}\),
      \({[\Qf{1},\Qf{n}]}\),
      \({[\Qf{2},\Qf{n}]}\),
      \({\vphantom{\Qf{1}}\smash{\vdots}\;\;}\),
      \({[\Qf{n},\Qf{n}]}\)
    }
    {
      \node[left=3.5ex] at (p-\i-1) {\v};
    }
    \foreach[count = \i] \v in {
      \({[\Qf1,\Qf1]}\),
      \({\cdots}\),
      \({[\Qf{1},\Qf{n}{-}\Qf{1}]}\),
      \({[\Qf{1},\Qf{n}]}\),
      \({[\Qf{2},\Qf{n}]}\),
      \({\cdots}\),
      \({[\Qf{n},\Qf{n}]}\)
    }
    {
      \node[above=1.5ex] at (p-1-\i) {\v};
    }
  }.
\]

At this point, we note that we can perform operations on
(the rows and columns of)
the group of blocks above-left of $E$
and on (the rows and columns of)
the group of blocks below-right of $E$
without affecting the other group (i.e.\ independently of each other).
Thus, we consider the subproblem
\[
  \mattikz{
    \matheaderwide{
      *           \& \cdots \& *      \\
      \vdots      \& \ddots \& \vdots \\
      *           \& \cdots \& *      \\
    };
    \node[anchor = south east] (p-0-0) at (p-2-1.west |- p-1-1.north west) {};
    \foreach[count = \i] \v in {
      \({[\Qf1,\Qf1]}\),
      \({\vphantom{\Qf{1}}\smash{\vdots}\;\;}\),
      \({[\Qf{1},\Qf{n}{-}\Qf{1}]}\)
    }
    {
      \node[left=3.5ex] at (p-\i-1) {\v};
    }
    \foreach[count = \i] \v in {
      \({[\Qf1,\Qf1]}\),
      \({\cdots}\),
      \({[\Qf{1},\Qf{n}{-}\Qf{1}]}\)
    }
    {
      \node[above=1.5ex] at (p-1-\i) {\v};
    }
  }
\]
where left-to-right block operations and below-to-above block operations
are permissible (Definition~\ref{defn:permissible}),
and
where by Lemma~\ref{lem:invertible} the entire matrix is invertible.
In Definition~\ref{defn:permissible}, elementary row (resp.\ column) operations
among rows (resp.\ columns) with the same label are all permissible,
but we do not need all of them (see the next Lemma).

The following is from basic linear algebra.
\begin{lemma}
  \label{lem:topermutation}
  Let $A$ be an invertible matrix.
  Then, using only the elementary operations of 
  \begin{enumerate}
  \item multiplying a row (or column) by a nonzero scalar,
  \item adding a multiple of a row to a row above it, and
  \item adding a multiple of a column to a column to its right,
  \end{enumerate}
  $A$ can be transformed to a permutation matrix.
\end{lemma}
For completeness, we provide an algorithm in Algorithm~\ref{algo:topermutation}, which
we call \textproc{OneWayReduce} because only operations in one direction (one-way) are allowed.
\begin{algorithm}[H]
  \caption{Reduction of an invertible $d\times d$ matrix $A$}
  \label{algo:topermutation}
  \begin{algorithmic}[1]
    \Procedure{OneWayReduce}{$A$}
    \For{$\texttt{col} \gets 1, 2, \hdots, d$}
    \State $\texttt{row} \gets \max\{i \mid A_{i, \texttt{col}} \neq 0 \}$ \label{line:topermutation_rowchoice}
    \State Multiply row $\texttt{row}$ by $(A_{\texttt{row}, \texttt{col}})^{-1}$
    \State Use entry at $(\texttt{row},\texttt{col})$ to zero out all entries to its right using column operations.
    \State Use entry at $(\texttt{row},\texttt{col})$ to zero out all entries above it using row operations.
    \State \Return $A$
    \EndFor
    \EndProcedure
  \end{algorithmic}
\end{algorithm}
In Algorithm~\ref{algo:topermutation}, the fact that $A$ is invertible guarantees that
the set $\{i \mid A_{i, \texttt{col}} \neq 0 \}$ in line~\ref{line:topermutation_rowchoice}
is always nonempty, and thus $\texttt{row}$ is well-defined (a valid row number).
Clearly, Algorithm~\ref{algo:topermutation} terminates,
and returns a matrix that satisfies the following property:
each column has exactly one entry ``$1$'', with all other entries ``$0$'',
and the ``$1$'' entries of different columns are never in the same row.
That is, a permutation matrix is returned.

The block matrix becomes
\[
  M' =
  \mattikz{
    \matheaderwide{
      \phantom{*}  \& \phantom{*}  \& \phantom{*} \& 0      \& \phantom{*} \& \phantom{*} \& \phantom{*} \\
      \phantom{*} \& \phantom{*}  \& \& \vdots \&             \&             \&  \\
      \phantom{*} \&  \& \phantom{*}      \& 0      \&             \&             \&  \\
      0           \& \cdots \& 0      \& E      \& 0           \& \cdots      \& 0 \\
      \phantom{*} \&        \&        \& 0      \& \phantom{*}            \&       \&  \\
      \phantom{*} \&        \&        \& \vdots \&       \& \phantom{*}       \&  \\
      \phantom{*} \&        \&        \& 0      \&            \&       \& \phantom{*}  \\
    };
    \node[anchor = south east] (p-0-0) at (p-2-1.west |- p-1-1.north west) {};
    \node[fit=(p-1-1.center)(p-3-3.center),draw,anchor=center,text depth=0em,minimum height=4.5em,minimum width=6em]{$\scalebox{2}{\ensuremath{P_\lefts}}$};
    \node[fit=(p-5-5.center)(p-7-7.center),draw,anchor=center,text depth=0em,minimum height=4.5em,minimum width=6em]{$\scalebox{2}{\ensuremath{P_\rights}}$};
    \foreach[count = \i] \v in {
      \({[\Qf1,\Qf1]}\),
      \({\vphantom{\Qf{1}}\smash{\vdots}\;\;}\),
      \({[\Qf{1},\Qf{n}{-}\Qf{1}]}\),
      \({[\Qf{1},\Qf{n}]}\),
      \({[\Qf{2},\Qf{n}]}\),
      \({\vphantom{\Qf{1}}\smash{\vdots}\;\;}\),
      \({[\Qf{n},\Qf{n}]}\)
    }
    {
      \node[left=3.5ex] at (p-\i-1) {\v};
    }
    \foreach[count = \i] \v in {
      \({[\Qf1,\Qf1]}\),
      \({\cdots}\),
      \({[\Qf{1},\Qf{n}{-}\Qf{1}]}\),
      \({[\Qf{1},\Qf{n}]}\),
      \({[\Qf{2},\Qf{n}]}\),
      \({\cdots}\),
      \({[\Qf{n},\Qf{n}]}\)
    }
    {
      \node[above=1.5ex] at (p-1-\i) {\v};
    }
  }
\]
where $P_\lefts$ and $P_\rights$ are permutation matrices.
Furthermore, by recording the operations performed on the original block matrix $M(\phi, \psi)$,
this can be expressed as
\begin{equation}
  \label{eq:matrixtrans}
  M' = T' M(\phi,\psi) T
\end{equation}
for some invertible block upper-triangular matrices $T$ and $T'$ of the forms
\begin{equation}
  \label{eq:Uform}
  \scalebox{0.75}{
  \mattikz{
    \matheaderwide{
      *  \& \cdots  \& {*} \& 0      \& \phantom{*} \& \phantom{*} \& \phantom{*} \\
       \& \ddots  \& \vdots \& \vdots \&             \&             \&  \\
      0 \&  \& *      \& 0      \&             \&             \&  \\
      0           \& \cdots \& 0      \& *     \& *           \& \cdots      \& * \\
      \phantom{*} \&        \&        \& 0      \& {*}    \& \cdots      \& * \\
      \phantom{*} \&        \&        \& \vdots \&       \& \ddots       \& \vdots \\
      \phantom{*} \&        \&        \& 0      \& 0         \&       \& *  \\
    };
    \node[anchor = south east] (p-0-0) at (p-2-1.west |- p-1-1.north west) {};
    \foreach[count = \i] \v in {
      \({[\Qf1,\Qf1]}\),
      \({\vphantom{\Qf{1}}\smash{\vdots}\;\;}\),
      \({[\Qf{1},\Qf{n}{-}\Qf{1}]}\),
      \({[\Qf{1},\Qf{n}]}\),
      \({[\Qf{2},\Qf{n}]}\),
      \({\vphantom{\Qf{1}}\smash{\vdots}\;\;}\),
      \({[\Qf{n},\Qf{n}]}\)
    }
    {
      \node[left=3.5ex] at (p-\i-1) {\v};
    }
    \foreach[count = \i] \v in {
      \({[\Qf1,\Qf1]}\),
      \({\cdots}\),
      \({[\Qf{1},\Qf{n}{-}\Qf{1}]}\),
      \({[\Qf{1},\Qf{n}]}\),
      \({[\Qf{2},\Qf{n}]}\),
      \({\cdots}\),
      \({[\Qf{n},\Qf{n}]}\)
    }
    {
      \node[above=1.5ex] at (p-1-\i) {\v};
    }
  }
  }
  \text{ and }
  \scalebox{0.75}{
  \mattikz{
    \matheaderwide{
      *  \& \cdots  \& {*} \& *      \& \phantom{*} \& \phantom{*} \& \phantom{*} \\
       \& \ddots  \& \vdots \& \vdots \&             \&             \&  \\
      0 \&  \& *      \& *      \&             \&             \&  \\
      0           \& \cdots \& 0      \& *     \& 0           \& \cdots      \& 0 \\
      \phantom{*} \&        \&        \& 0      \& {*}    \& \cdots      \& * \\
      \phantom{*} \&        \&        \& \vdots \&       \& \ddots       \& \vdots \\
      \phantom{*} \&        \&        \& 0      \& 0         \&       \& *  \\
    };
    \node[anchor = south east] (p-0-0) at (p-2-1.west |- p-1-1.north west) {};
    \foreach[count = \i] {\v} in {
      \({[\Qf1,\Qf1]}\),
      \({\vphantom{\Qf{1}}\smash{\vdots}\;\;}\),
      \({[\Qf{1},\Qf{n}{-}\Qf{1}]}\),
      \({[\Qf{1},\Qf{n}]}\),
      \({[\Qf{2},\Qf{n}]}\),
      \({\vphantom{\Qf{1}}\smash{\vdots}\;\;}\),
      \({[\Qf{n},\Qf{n}]}\)
    }
    {
      \node[left=3.5ex] at (p-\i-1) {\v};
    }
    \foreach[count = \i] \v in {
      \({[\Qf1,\Qf1]}\),
      \({\cdots}\),
      \({[\Qf{1},\Qf{n}{-}\Qf{1}]}\),
      \({[\Qf{1},\Qf{n}]}\),
      \({[\Qf{2},\Qf{n}]}\),
      \({\cdots}\),
      \({[\Qf{n},\Qf{n}]}\)
    }
    {
      \node[above=1.5ex] at (p-1-\i) {\v};
    }
  }
}
\end{equation}
respectively.

\subsection{From the normal form to the interval decomposition}
\label{subsec:normal_form_to_decompo}
The form of $M'$ gives matchings:
\begin{enumerate}
\item between intervals $[\Qf{1},?]$ of the upper and lower path,
\item between intervals $[?,\Qf{n}]$ of the upper and lower path,
\item between intervals $[\Qf{1},\Qf{n}]$ of the upper and lower path,
\end{enumerate}
given by the permutation matrices $P_\lefts$, $P_\rights$, and the identity matrix $E$, respectively.
We show that by ``gluing'' these matched intervals,
and together with the intervals not intersecting $\Qf{1}$ nor $\Qf{n}$
(the terms $\upperX_\others$ and $\lowerX_\others$ in the partition of the intervals of $\upperX$ and $\lowerX$),
we obtain a bipath persistence module $W$ that is explicitly a direct sum of interval bipath persistence modules
and is isomorphic to the original bipath persistence module $V$.
For computations, we only need to do so (i.e.\ construct/write out the intervals of $W$).
Below, for convenience, we summarize this and explicitly write out what we mean by ``gluing''.
Note that we return to the original indexing vertices
$\{\hat0, 1,\hdots,i,\hdots,n,\hat1\}$ for the upper path $U$
and
$\{\hat0, 1',\hdots,j',\hdots,m',\hat1\}$ for the lower path $D$.
We also recall Definition~\ref{def:LRUD} for intervals in the bipath poset.
\begin{construction}
  \label{construction:intervals}
  Given the considerations of the above discussion, construct the following intervals of the bipath poset $B$ (allowing for multiplicities).
  \begin{enumerate}
  \item For each matched pair $[\hat0, \hat1]$ in $U$ and $[\hat0, \hat1]$ in $D$,
    construct the interval $B$ (i.e.\ the bipath poset itself).
    These correspond to region $\mathrm{B}$ in Figure~\ref{fig:bipathPD}.
  
    \item For each matched pair $[\hat0, i]$ in $U$ and $[\hat0, j']$ in $D$,
    construct the interval $\langle j', i \rangle\in \mathbb{L}(B)$.
    These correspond to intervals in region $\mathbb{L}$ of Figure~\ref{fig:bipathPD}. 
  \item For each matched pair $[i, \hat1]$ in $U$ and $[j', \hat1]$ in $D$,
    construct the interval $\langle i, j' \rangle\in \mathbb{R}(B)$.
    These correspond to intervals in region $\mathbb{R}$ of Figure~\ref{fig:bipathPD}.
  \item For each interval $[i_1,i_2]$ in $\upperX_\others$ (intervals of $U$ not intersecting $\hat0$ nor $\hat1$),
    construct the interval $\langle i_1, i_2 \rangle\in \mathbb{U}(B)$.
    These correspond to intervals in region $\mathbb{U}$ of Figure~\ref{fig:bipathPD}.
  \item For each interval $[j'_1,j'_2]$ in $\lowerX_\others$ (intervals of $D$ not intersecting $\hat0$ nor $\hat1$),
    construct the interval $\langle j'_2, j'_1 \rangle\in \mathbb{D}(B)$.
    These correspond to intervals in region $\mathbb{D}$ of Figure~\ref{fig:bipathPD}. 
  \end{enumerate}
\end{construction}
In the rest of this subsection, we provide the proof that this is algebraically valid.

The form of the invertible block upper triangular matrix $T$ as given above
in Equation~\eqref{eq:Uform} has only nonzero entries
in the block located in row $\Intp{a,b}$ and column $\Intp{c,d}$
only for pairs $\Intp{a,b} \opreltoeq \Intp{c,d}$.
By attaching the fixed morphisms
$
  f\itoi{a,b,c,d} 
\colon k_{\Intp{a,b}} \rightarrow k_{\Intp{c,d}}
$
to the matrix $T$ (i.e.\ use the entries of $T$ as coefficients),
we obtain an endomorphism of
$\upperX_\lefts \oplus \upperX_\centers \oplus \upperX_\rights$.
This is in fact an automorphism, because $T$ is invertible block upper-triangular.
Furthermore, taking a direct sum with identity maps for the terms $\upperX_\others$,
we obtain an automorphism for $\upperX = \upperX_\lefts \oplus \upperX_\centers \oplus \upperX_\rights \oplus \upperX_\others$, which we denote by
\[
  \hat{T} \colon \upperX \rightarrow \upperX.
\]
Similiarly, from $T'$ we obtain an automorphism
\[
  \hat{T'} \colon \lowerX \rightarrow \lowerX.
\]
Then, by construction,
$M' = T' M(\phi,\psi) T$
is a block matrix problem (Definition~\ref{defn:matrixproblem})
associated with the pair of isomorphisms
$(\phi', \psi') := (\hat{T}^{-1}\phi, \hat{T'} \psi)$:
\begin{eqnarray*}
  \begin{tikzcd}
    V_\upperp \rar{\phi} & \upperX \rar{\hat{T}^{-1}} & \upperX
  \end{tikzcd}
  & (\text{in } \rep (\upperp)), \\
  \begin{tikzcd}
    V_\lowerp \rar{\psi} & \lowerX \rar{\hat{T'}} & \lowerX
  \end{tikzcd}
  & (\text{in } \rep (\lowerp)).
\end{eqnarray*}

Recalling the poset $B'$ in Diagram~\eqref{diag:extendedbipath},
the inclusion functor $\iota \colon  \rep (B) \rightarrow \rep (B')$
and Lemma~\ref{lem:formalism1},
$\phi'$ and $\psi'$ induces an isomorphism $\iota(V) \rightarrow Z'$
where $Z'$ is the representation

\[
  \begin{tikzcd}[ampersand replacement=\&]
    \upperX(\hat{0}) \dar[swap]
    {\Lambda' :=      
      \smat{
        P_\lefts  & 0 \\
        0 & E
      }}
    \rar \& \upperX(1) \rar \& \cdots \rar \& \upperX(n) \rar  \& \upperX(\hat{1})
    \dar
    {\Gamma' :=
      \smat{
        E & 0 \\
        0 & P_\rights
      }} \\
    \lowerX(\hat{0}') \rar \& \lowerX(1') \rar \& \cdots \rar \& \lowerX(m') \rar  \&  \lowerX(\hat{1}')
  \end{tikzcd}
\]
  by construction (i.e.\ from the normal form of the block matrix problem).

  By permuting the summands, say of $X$ (write the permuted representation as $X'$), we obtain an isomorphism $Z' \cong Z''$
  where the ``vertical'' maps of $Z''$ are identity, that is
\[
  Z'':
  \begin{tikzcd}[ampersand replacement=\&]
    \upperX'(\hat{0}) \dar[swap]
    {\Lambda'' := 1}
    \rar \& \upperX'(1) \rar \& \cdots \rar \& \upperX'(n) \rar  \& \upperX'(\hat{1})
    \dar
    {\Gamma' := 1} \\
    \lowerX(\hat{0}') \rar \& \lowerX(1') \rar \& \cdots \rar \& \lowerX(m') \rar  \&  \lowerX(\hat{1}')
  \end{tikzcd}.
\]
  By the definition of the inclusion functor $\iota$, such a representation $Z''$ can be written as
  $Z'' = \iota(W)$ where $W \in \rep (B)$ is:
  \[    
    \begin{tikzcd}[row sep=0.1em, inner sep=0pt]
      & \upperX'(1) \rar & \cdots \rar & \upperX'(n) \rar  & \upperX'(\hat{1}) \ar[dr] & \\
      \left(\upperX'(\hat{0}) = \lowerX(\hat0)\right) \ar[ur] \ar[dr] & & & & & \left(\upperX'(\hat{1}) = \lowerX(\hat1)\right) \\    
      & \lowerX(1') \rar & \cdots \rar & \lowerX(m') \rar  &  \lowerX(\hat{1}') \ar[ur]&
    \end{tikzcd},
  \]
  i.e.\ defined by the restriction of $Z''$ to $B$.
  Combining, we obtain
  \[
    \iota(V) \cong Z' \cong Z'' = \iota(W)
  \]
  in $\rep (B')$.
  By Lemma~\ref{lem:formalism0}, we conclude that 
  \[
    V \cong W
  \]
  in $\rep (B)$, and we note that $W$ is (explicitly) a direct sum of intervals of $\rep(B)$,
  since both $\upperX'$ and $\lowerX$ are direct sums of intervals (in $\rep (\upperp)$ and $\rep (\lowerp)$ respectively).
  That is, we have computed an interval decomposition of $V$, corresponding to Construction~\ref{construction:intervals}.
  \subsection{Summary of the algorithm and an algorithm for bipath filtrations}
  \label{subsec:algosummary}
We summarize the entire procedure below as Algorithm~\ref{algo:main}.
\begin{algorithm}[H]
  \caption{Main Algorithm}
  \label{algo:main}
  \begin{algorithmic}[1]
    \Require bipath persistence module $V$
    \Procedure{IntervalDecompose}{$V$}    
    \State Interval-decompose the upper row $V_\upperp$ and lower row $V_\lowerp$.
    \State Construct the associated block matrix problem (Definition~\ref{defn:matrixproblem}).
    \State Reduce the block matrix problem to normal form (Subsection~\ref{subsec:normalform}).    
    \State \Return The intervals of Construction~\ref{construction:intervals}
    \EndProcedure
  \end{algorithmic}
\end{algorithm}

The above starts with a bipath persistence module $V$ as its input.
What if we are starting with a bipath filtration?
Consider the bipath filtration
\begin{equation*}
  S:
  \begin{tikzcd}[row sep=0.1em,column sep = 1.4em, inner sep=0pt]
    & S_1 \rar[hookrightarrow] & S_2 \rar[hookrightarrow] & \cdots \rar[hookrightarrow] & S_n \ar[dr,hookrightarrow] & \\
    S_{\hat0} \ar[ur,hookrightarrow] \ar[dr,hookrightarrow] & & & &  & S_{\hat1} \\
    & S_{1'} \rar[hookrightarrow] & S_{2'} \rar[hookrightarrow] & \cdots \rar[hookrightarrow] & S_{m'} \ar[ur,hookrightarrow] &
  \end{tikzcd}.
\end{equation*}
Applying the homology functor $H_q(-;k)$ we obtain the bipath persistence module $V = H_q(S;k)$: 
\begin{equation*}  
  \begin{tikzcd}[row sep=0.1em,column sep = 1.4em, inner sep=0pt]
    & H_q(S_1;k) \rar[] & H_q(S_2;k) \rar[] & \cdots \rar[] & H_q(S_n;k) \ar[dr] & \\
    H_q(S_{\hat0};k) \ar[ur] \ar[dr] & & & &  & H_q(S_{\hat1};k) \\
    & H_q(S_{1'};k) \rar[] & H_q(S_{2'};k) \rar[] & \cdots \rar[] & H_q(S_{m'};k) \ar[ur] &
  \end{tikzcd}
\end{equation*}
whose interval decomposition we want to compute.
Instead of explicitly computing $V$, say by pointwise computing homology vector spaces and the induced linear maps between them, we can instead apply the
standard persistent homology algorithm to
the two filtrations
\begin{equation*}    
S_\upperp\colon \  S_{\hat0} \subseteq S_1 \subseteq S_2 \subseteq \cdots
\subseteq S_n \subseteq S_{\hat1}  
\end{equation*}
and 
\begin{equation*}    
S_\lowerp\colon \  S_{\hat0} \subseteq S_{1'} \subseteq S_{2'} \subseteq \cdots
\subseteq S_{m'} \subseteq S_{\hat1}  
\end{equation*}
separately.

\begin{remark}
\label{rem:bipathfilt_const}
In addition to the intervals of the upper and lower rows,
the matrices $\Lambda$ and $\Gamma$  are needed to construct the associated block matrix problem.
For this, we do not need to explicitly compute the isomorphisms $\phi$ and $\psi$.
Instead, it suffices to record a set of representative cycles
associated to the intervals $[b,d]$ of the persistence diagrams of the upper and lower rows
(in fact, only the ones associated to the intervals $[\hat0, ?]$ and $[?, \hat1]$ are needed) together with bases for the image of the boundary at $\hat0$ and $\hat1$.
Note that the standard reduction algorithm for persistent homology
provides these information.
From these, one can compute the matrices $\Lambda$ (resp., $\Gamma$)
as change of basis matrices for the basis at $\hat0$ (resp., $\hat1$)
determined by (homology classes of) the representative cycles
of the upper row versus the one determined by the
representative cycles of the lower row.
\end{remark}

We summarize this in Algorithm~\ref{algo:main2}.
\begin{algorithm}[H]
  \caption{Main Algorithm -- for bipath filtrations}
  \label{algo:main2}
  \begin{algorithmic}[1]
    \Require Bipath filtration $S$
    \Procedure{IntervalDecompose}{$S,q$}    
    \State Compute the standard $q$th persistence diagram of $S_\upperp$ and $S_\lowerp$
    \StatexIndent[2] (with additional information noted in Remark~\ref{rem:bipathfilt_const}).
    \State Construct the associated block matrix problem with $\Lambda$ and $\Gamma$ as in Remark~\ref{rem:bipathfilt_const}.
    \State Reduce the block matrix problem to normal form (Subsection~\ref{subsec:normalform}).    
    \State \Return The intervals of Construction~\ref{construction:intervals}   
    \EndProcedure
  \end{algorithmic}
\end{algorithm}

Our main theorem is as follows (which is proven through this Section).
\begin{theorem}
  Given a bipath persistence module $V$ (resp., bipath filtration $S$),
  Algorithm~\ref{algo:main} (resp., Algorithm~\ref{algo:main2})
  correctly computes its interval decomposition (resp., $q$th bipath persistence diagram).
\end{theorem}

\section{Interpretation of bipath persistence diagrams}\label{sec:bipathPD}
  In this section, we illustrate the decomposition of a bipath persistence module obtained from a small example of a bipath filtration of simplicial complexes.
  In addition, we visualize the bipath persistence module (see subsection~\ref{sec:bipath_pd}) and its persistence diagram
  and give an interpretation for the visualization.

Consider an abstract simplicial complex $\Delta$ whose $j$-faces are given by 
$$
\Delta^0 :=
\{a,b,c,d,e\}, \ 
\Delta^1 := \{\{a,b\}, \{a,c\}, \{a,d\}, \{a,e\}, \{b,c\},\{c,d\},\{c,e\},\{d,e\}\}, \ 
\Delta^2 := \{\{a,c,e\}\}, 
$$
and a $B_{3,2}$-filtration 
\begin{equation*}
  S\colon \ 
\begin{tikzcd}[row sep=0.1em,column sep = 1.4em, inner sep=0pt]
    & S_{1} \ar[r,hookrightarrow] & S_{2} \ar[r,hookrightarrow] & S_{3} \ar[dr,hookrightarrow] & \\ 
    S_{\hat0} \ar[ur,hookrightarrow] \ar[dr,hookrightarrow] & & & & S_{\hat1} \\
    & S_{1'} \ar[rr, hookrightarrow] & & S_{2'} \ar[ur,hookrightarrow] & 
\end{tikzcd} 
\end{equation*}
for $\Delta$ defined as follows 
(This is illustrated in Figure~\ref{fig:BFSC}): 
Let $S_{\hat{0}} := \Delta^0\cup \{\{a,b\}, \{b,c\}\}$ and $S_{\hat{1}} := \Delta$. 
In addition, let 
\begin{eqnarray*}
    S_{1} &:=& S_{\hat0} \cup \{\{a,c\},\{a,d\},\{a,e\},\{d,e\}\}, \\
    S_{2} &:=& S_{1} \cup \{\{c,d\}\},  \\
    S_{3} &:=& S_{2} \cup \{\{c,e\} \}, \\
    S_{1'} &:=& S_{\hat0} \cup \{\{a,c\},\{a,e\},\{c,e\} \},\\
    S_{2'} &:=&S_{1'} \cup \{\{a,d\},\{d,e\},\{a,c,e\}\}.  
\end{eqnarray*}

\begin{figure}[th]
\begin{center}
\includegraphics[width=90mm]{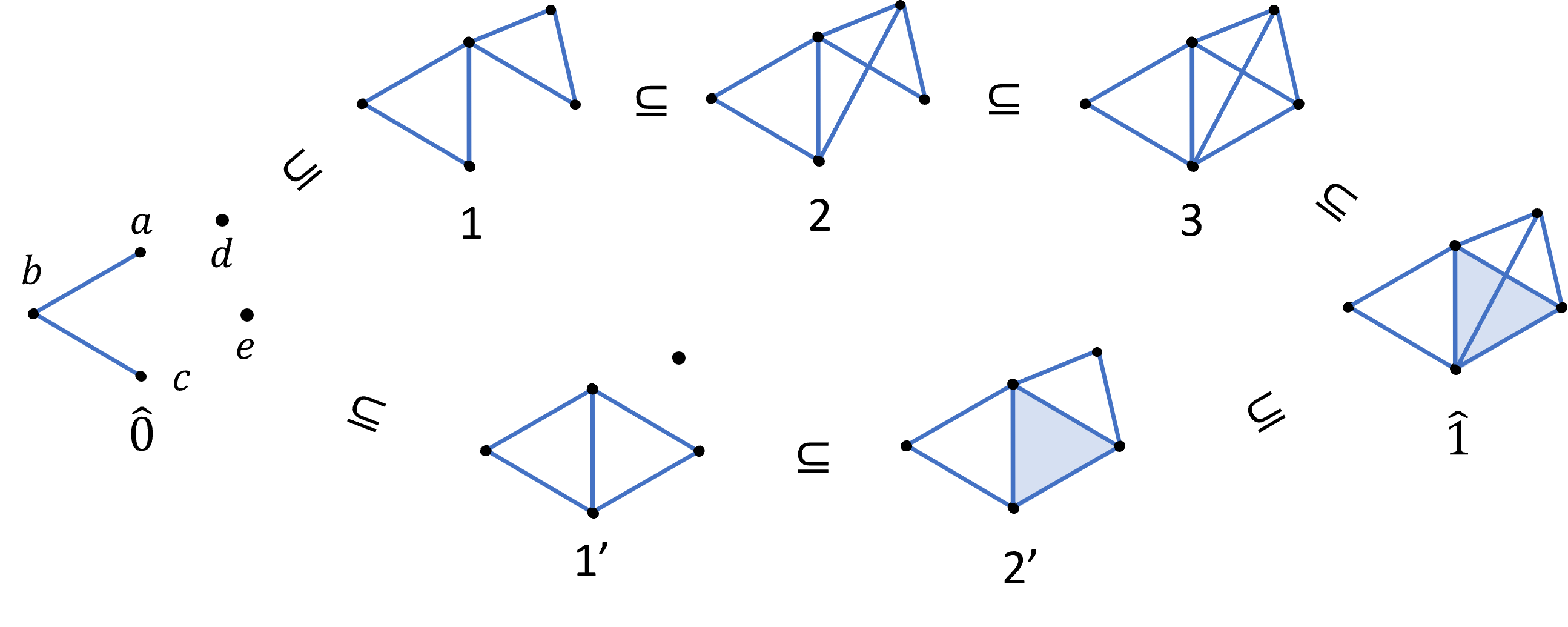}
\caption{A geometric realization of a bipath filtration $S$.}
\label{fig:BFSC}
\end{center}
\end{figure}

We apply the $q$th homology functor with coefficient field $k=\mathbb{F}_2$ (a field with two elements) to $S$ and obtain a bipath persistence module $V_q := H_q(S;\mathbb{F}_2)$ for each $q$. 
\begin{equation*} 
\begin{tikzcd}[row sep=0.1em,column sep = 1.4em, inner sep=0pt]
    & H_q(S_{1};\mathbb{F}_2) \ar[r] & H_q(S_{2};\mathbb{F}_2) \ar[r] & H_q(S_{3};\mathbb{F}_2) \ar[dr] & \\ 
    H_q(S_{\hat0};\mathbb{F}_2) \ar[ur] \ar[dr] & & & & H_q(S_{\hat1};\mathbb{F}_2). \\
    & H_q(S_{1'};\mathbb{F}_2) \ar[rr] & & H_q(S_{2'};\mathbb{F}_2) \ar[ur] & 
\end{tikzcd} 
\end{equation*}
For computing its persistence diagrams, let us decompose it into interval modules by using Algorithm~\ref{algo:main}.

We clearly have $V_q = 0$ for all $q\geq 2$. 
For $q\in \{0,1\}$, let $(V_{q})_U$ and $(V_{q})_D$ be restrictions of $V_q$ to the upper path $U$ and lower path $D$ respectively. 
Then, we have persistence modules 

\begin{equation*}
     X_q = \begin{cases}
        k_{[\hat{0},\hat{0}]}\oplus k_{[\hat{0},\hat{0}]}\oplus k_{[\hat{0},\hat{1}]} & \text{($q=0$),}\\ 
         k_{[1,\hat{1}]}\oplus k_{[1,\hat{1}]}\oplus k_{[2,\hat{1}]} \oplus k_{[3,3]}& \text{($q=1$),} 
    \end{cases}
    \ \ \text{and} \ \ 
    Y_q = \begin{cases}
        k_{[\hat{0},\hat{0}]}\oplus k_{[\hat{0},1']}\oplus k_{[\hat{0},\hat{1}]} & \text{($q=0$),} \\ 
         k_{[1',1']}\oplus k_{[1',\hat{1}]}\oplus k_{[2',\hat{1}]} \oplus k_{[\hat{1},\hat{1}]}& \text{($q=1$),} 
    \end{cases}
\end{equation*}
over $U$ and $D$ respectively, 
and a pair of isomorphisms $\phi_q \colon (V_q)_U \to X_q$ and $\psi_q \colon (V_q)_D \to Y_q$ such that 
they define the block matrix problem (Definition~\ref{defn:matrixproblem}) associated with $(\phi_q,\psi_q)$ of the following form for $q=0$ and $q=1$ respectively. 
\begin{equation*}
  \mattikz{
    \matheaderwide{
      1    \& 0 \& 0 \\
      0    \& 1 \& 0 \\
      0    \& 0 \& 1 \\
    };
    \node[anchor = south east] (p-0-0) at (p-2-1.west |- p-1-1.north west) {};
    \foreach[count = \i] \v in {
      \({[\hat{0},\hat{0}]}\),
     \({[\hat{0},1']}\),
      \({[\hat{0},\hat{1}]}\)
    }
    {
      \node[left=3.5ex] at (p-\i-1) {\v};
    }
    \foreach[count = \i] \v in {
      \({[\hat{0},\hat{0}]}\),
      \({[\hat{0},\hat{0}]}\),
      \({[\hat{0},\hat{1}]}\)
    }
    {
      \node[above=1.5ex] at (p-1-\i) {\v};
    }
  }
  \text{ \ and \ }
    \mattikz{
    \matheaderwide{
      1    \& 0 \& 0 \\
      0    \& 1 \& 1 \\
      0    \& 0 \& 1 \\
    };
    \node[anchor = south east] (p-0-0) at (p-2-1.west |- p-1-1.north west) {};
    \foreach[count = \i] \v in {
        \({[1',\hat{1}]}\),
        \({[2',\hat{1}]}\),
        \({[\hat{1},\hat{1}]}\)
    }
    {
      \node[left=3.5ex] at (p-\i-1) {\v};
    }
    \foreach[count = \i] \v in {
      \({[1,\hat{1}]}\),
      \({[1,\hat{1}]}\),
      \({[2,\hat{1}]}\)
    }
    {
      \node[above=1.5ex] at (p-1-\i) {\v};
    }
  }.
\end{equation*}
Solving this (i.e., computing its normal form) for each $q$, we obtain a collection of intervals of $B_{3,2}$ 
by Construction~\ref{construction:intervals}, 
which corresponds to the interval-decomposition of $V_q$. 
We conclude that 

\begin{equation*}
    V_q \cong \begin{cases}
        k_{B_{3,2}} \oplus 
        k_{\langle \hat{0},\hat{0}\rangle} \oplus  
        k_{\langle 1',\hat{0} \rangle} & 
        \text{($q=0$),} \\ 
        k_{\langle 1,1'\rangle} \oplus k_{\langle 1,2'\rangle} \oplus k_{\langle 2,\hat{1}\rangle} \oplus k_{\langle 3,3\rangle} \oplus k_{\langle 1',1'\rangle} & \text{($q=1$),} \\ 
        0 & \text{($q\geq 2$),} 
    \end{cases}
\end{equation*}
and 
\begin{equation}
\label{eq:persdigmS}
    \mathcal{D}_q(S) = \begin{cases}
        \left\{B_{3,2}, 
        \langle \hat{0},\hat{0} \rangle,
        \langle 1',\hat{0}\rangle \right\} & \text{($q=0$),}\\ 
        \left\{
        \langle 1,1'\rangle, 
        \langle 1,2'\rangle, 
        \langle 2,\hat{1}\rangle, 
        \langle 3,3\rangle, 
        \langle 1',1'\rangle 
        \right\} & \text{($q=1$),} \\ 
        \emptyset & \text{($q\geq 2$).}
    \end{cases}
\end{equation}
We describe the above bipath persistence diagrams in Figure~\ref{fig:0and1}.

We remark that we may use Algorithm~\ref{algo:main2} for a bipath filtration $S$ to obtain the persistence diagrams without explicitly computing $H_q(S;\mathbb{F}_2)$. 
Of course, it returns the same result as \eqref{eq:persdigmS}. 

\begin{figure}[th]
\begin{center}
\includegraphics[width=100mm]{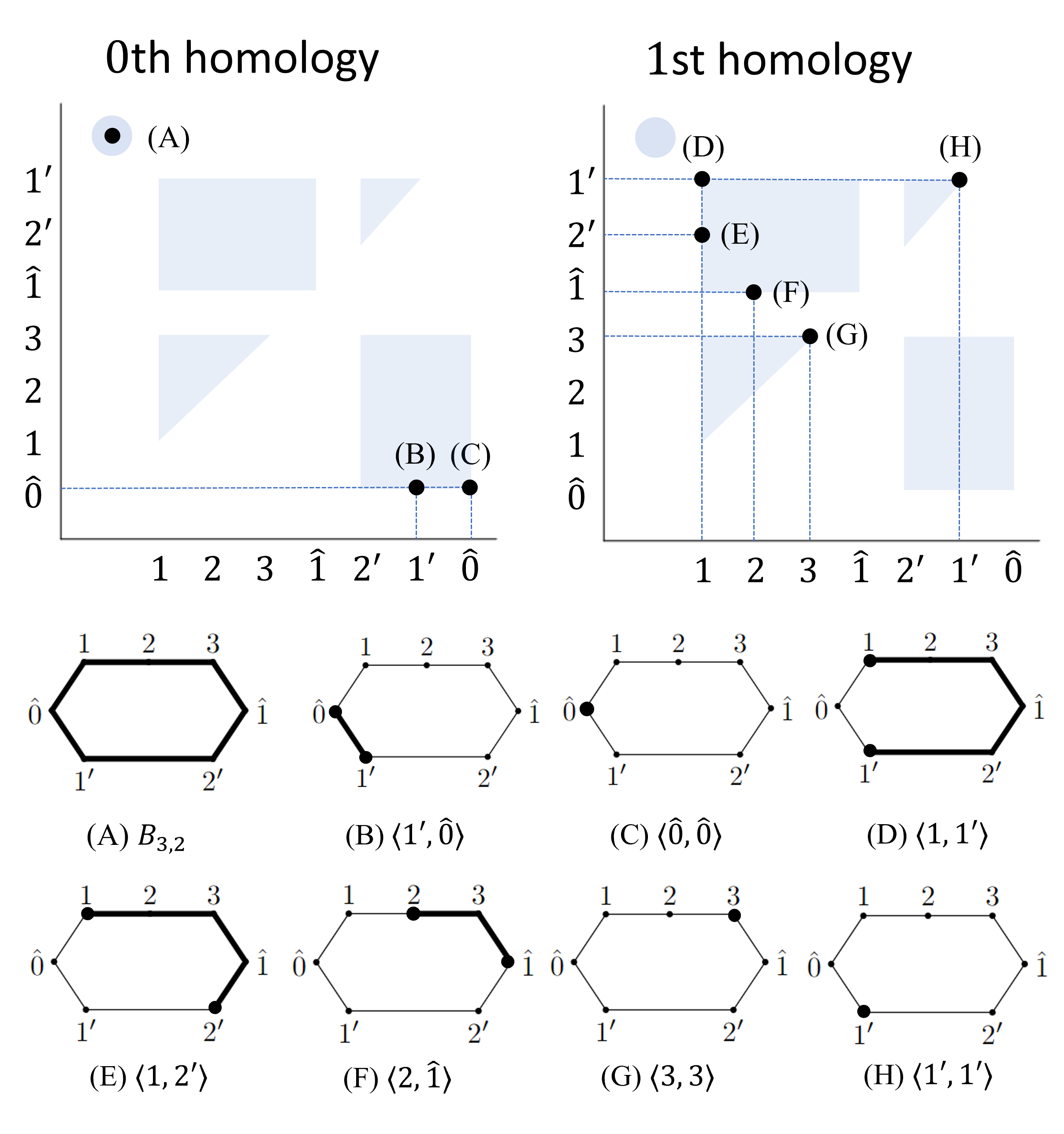}
\caption{
A visualization of $\mathcal{D}_q(S)$ for $q\in \{0,1\}$. Each point has the multiplicity one in this case.
}
\label{fig:0and1}
\end{center}
\end{figure}

Below, we give an interpretation for the persistence diagram, using some geometric intuition.
However, we warn that this is dependent on a choice of representative cycles for the homology classes,
and thus the interpretation (via representative ``holes'') is not unique, in principle.
\begin{itemize}
\item For $q=0$, every connected component is born at $\hat{0}$. 
  The interval $B_{3,2}$ corresponds to a connected component containing $a$ which is alive until the end. 
  Similarly, the interval $\langle 1',\hat{0}\rangle$ (resp., $\langle\hat{0},\hat{0}\rangle$) 
  corresponds to the connected component of $d$ (resp., $e$)
  which dies (gets connected to another connected component) at $1'$ (resp., $\hat{0}$). 
  These are illustrated by the points (A), (B), and (C) in Figure~\ref{fig:0and1}.
  
\item
  For $q=1$, the interval $\langle 1,1'\rangle$ corresponds to the hole formed by $a,b,c$,
  which is born at both $1$ (upper path) and $1'$ (lower path)
  and is alive until the end in both paths, illustrated by point (D) in Figure~\ref{fig:0and1}.  
  Similarly, the interval $\langle 1,2' \rangle$  corresponds to a hole formed by $a,d,e$, born at $1$ and $2'$
  and is alive until the end in both paths; this is point (E) in Figure~\ref{fig:0and1}. 

  The interval $\langle 2,\hat1 \rangle$  corresponds to a hole formed by $a,c,d$, which is born at $2$ (upper path)
  and $\hat1$ (lower path)  and is alive until the end $\hat1$.
  Note that it is only present in the lower path at $\hat1$. This corresponds to point (F) in Figure~\ref{fig:0and1}.
  
  Both the interval $\langle 3,3 \rangle$ (point (G) in Figure~\ref{fig:0and1})
  and the interval $\langle 1',1' \rangle$ (point (H) in Figure~\ref{fig:0and1})
  can be interpreted as
  corresponding to the hole formed by $a,c,e$,
  which is born and dies at $3$ and at $1'$ in the upper and lower path.
  Since they are not connected at the endpoints $\hat0$ and $\hat1$, these are two different intervals.
\end{itemize}

  One motivation for our introduction of bipath persistence is that it allows us to study the persistence of topological features across a pair of filtrations to compare the two filtrations.
  In the above example, the intervals (A) and (D) point to the existence of topological features
  persisting for long across both upper and lower filtrations,
  while intervals (B), (C), (G), and (H) may be seen as noise
  because of their short persistence.
  The intervals (E) and (F) correspond to topological features persisting
  for longer across the upper filtration than the lower filtration.
  This way, we can see the differences and compare topological features across two filtrations.


\section*{Acknowledgements}
The authors would like to thank 
  Tomoyuki Shirai for discussions concerning
  the graphical display of the bipath persistence diagrams.
This work is supported by JSPS Grant-in-Aid for Transformative Research Areas (A) (22H05105).
S.T. is supported by JST SPRING, Grant Number JPMJSP2148.

\bibliographystyle{alpha} 
\bibliography{decomposition.bib}

\end{document}